\documentclass[a4paper,11pt]{amsart}
\pdfoutput=1

\usepackage[utf8]{inputenc}
\usepackage[T1]{fontenc}
\usepackage[
  text={445pt,575pt},
  headheight=9pt,
  centering
]{geometry}

\usepackage{microtype}
\usepackage{amsfonts,amssymb}
\usepackage{amsthm,amsmath}
\usepackage{xspace}
\usepackage[skip=.25\baselineskip]{parskip}
\usepackage{mathtools}
\usepackage{textcomp}
\usepackage{enumitem}
\usepackage{tikz-cd}
\usepackage{bbm}
\makeatletter
\g@addto@macro\bfseries{\boldmath}
\makeatother
\usepackage[frozencache, cachedir=minted]{minted}
\setminted[macaulay2]{mathescape=true}

\usepackage{todonotes}
\makeatletter
\newcommand{\DefineTodoColor}[2]{%
  \define@key{todonotes}{#1}[]{\gdef\@todonotes@currentbackgroundcolor{#2}}%
}
\makeatother

\DefineTodoColor{Thomas}{red!20}
\DefineTodoColor{Frank}{blue!20}
\DefineTodoColor{Andreas}{green!20}
\DefineTodoColor{Tobias}{yellow!20}
\newlist{paraenum}{enumerate}{1}
\setlist[paraenum]{wide, label=(\arabic*)}
\newlist{inparaenum}{enumerate*}{1}
\setlist[inparaenum]{label=(\arabic*)}
\setlist[itemize]{leftmargin=4em,rightmargin=4.5em,label=---}
\setlist[enumerate]{leftmargin=4em,rightmargin=4.5em}
\setlist[description]{leftmargin=4em,labelindent=4em,rightmargin=4.5em}

\usepackage{nicematrix}
\NiceMatrixOptions{
  code-for-first-row = \scriptstyle,
  code-for-last-col  = \scriptstyle
}

\usepackage{graphicx}
\usepackage{url}
\usepackage{hyperref}
\usepackage[sort, capitalize]{cleveref}
\usepackage{xcolor}

\definecolor{darkred}{RGB}{160,0,0}
\definecolor{darkblue}{RGB}{0,0,160}
\hypersetup{
  colorlinks,
  citecolor=darkblue,
  filecolor=black,
  linkcolor=darkblue,
  urlcolor=darkblue
}
\theoremstyle{definition}
\newtheorem{theorem}{Theorem}[section]
\newtheorem{theoremAlph}{Theorem}

\newtheorem*{theorem*}{Theorem}

\newtheorem{lemma}[theorem]{Lemma}
\newtheorem{proposition}[theorem]{Proposition}
\newtheorem{corollary}[theorem]{Corollary}
\newtheorem{remark}[theorem]{Remark}

\newtheorem*{fact*}{Fact}
\newtheorem{example}[theorem]{Example}
\newtheorem{definition}[theorem]{Definition}

\newtheorem{conjecture}[theorem]{Conjecture}

\newtheorem{qu}[theorem]{Question}

\newcounter{q48}
\newcommand\defas{\coloneqq}

\newcommand\set[1]{\left\{#1\right\}}

\newcommand{\A}{\mathcal{A}}
\newcommand{\C}[1]{\mathcal{#1}}

\newcommand{\B}[1]{\mathbb{#1}}

\newcommand{\RR}{\B R}
\newcommand{\CC}{\B C}

\newcommand{\M}{\C M}

\newcommand{\PD}{\mathrm{PD}}
\newcommand{\GL}{\mathrm{GL}}

\newcommand{\Matus}{Mat{\'u}{\v{s}}\xspace}

\DeclareMathOperator{\adj}{adj}
\DeclareMathOperator{\tr}{tr}
\DeclareMathOperator{\rk}{rk}

\DeclareMathOperator{\Sym}{Sym}
\DeclareMathOperator{\codim}{codim}

\DeclareMathOperator{\edim}{edim}
\DeclareMathOperator{\CI}{CI}
\DeclareMathOperator{\SCI}{SCI}
\DeclareMathOperator{\sgn}{sgn}
\DeclareMathOperator{\Span}{span}
\DeclareMathOperator{\inv}{inv}
\DeclareMathOperator{\id}{id}

\makeatletter
\newsavebox{\@brx}
\newcommand{\llangle}[1][]{\savebox{\@brx}{\(\m@th{#1\langle}\)}%
  \mathopen{\copy\@brx\mkern2mu\kern-0.9\wd\@brx\usebox{\@brx}}}
\newcommand{\rrangle}[1][]{\savebox{\@brx}{\(\m@th{#1\rangle}\)}%
  \mathclose{\copy\@brx\mkern2mu\kern-0.9\wd\@brx\usebox{\@brx}}}
\makeatother

\newcommand{\CIS}[1]{\llangle #1\rrangle}
\newcommand{\ideal}[1]{\langle #1\rangle}
\newcommand{\dual}{\rceil}
\newcommand{\comp}{\mathsf{c}}
\newcommand{\ol}[1]{\overline{#1}}

\usepackage{listofitems}
\newcommand\Marg[1]{%
  \readlist*\mylist{#1}%
  \mylist[1] \mathbin{\backslash} \mylist[2]%
}
\newcommand\Cond[1]{%
  \readlist*\mylist{#1}%
  \mylist[1] \mathbin{/} \mylist[2]%
}
\newcommand\Perp{\protect\mathpalette{\protect\independenT}{\perp}}
\def\independenT#1#2{\mathrel{\rlap{$#1#2$}\mkern2mu{#1#2}}}
\newcommand{\ind}[2]{\left.#1 \Perp #2 \inD}
\newcommand{\inD}[1][\relax]{\def\argone{#1}\def\temprelax{\relax}
  \ifx\argone\temprelax\right.\else\,\middle|#1\right.{}\fi}

\usepackage{ifthen}
\usepackage{xparse}
\usepackage{tikz}
\usetikzlibrary{backgrounds}
\usetikzlibrary{arrows.meta}
\newcommand{\graphedgestyle}[1]{\tikzset{#1/.style={draw=black}}}
\NewDocumentCommand{\graph}{m >{\SplitList{,}}m O{scale=0.4}}{%
\begin{tikzpicture}[
  anchor = base, baseline,
  Eij/.style = {draw=none},
  Eik/.style = {draw=none},
  Eil/.style = {draw=none},
  Ejk/.style = {draw=none},
  Ejl/.style = {draw=none},
  Ekl/.style = {draw=none},
  #3
]

\tikzset{every node/.style={draw,shape=circle,fill=black,minimum size=2pt,inner sep=0}}
\tikzset{every edge/.style={line width=1pt}}
\ProcessList{#2}{\graphedgestyle}

\ifthenelse{#1 > 1}{
\node (Ni) at (0, 1) {};
\node (Nj) at (1, 1) {};

\path (Ni) edge [Eij] (Nj);
}{}

\ifthenelse{#1 > 2}{
\node (Nk) at (1, 0) {};

\path (Ni) edge [Eik] (Nk);
\path (Nj) edge [Ejk] (Nk);
}{}

\ifthenelse{#1 > 3}{
\node (Nl) at (0, 0) {};

\path (Ni) edge [Eil] (Nl);
\path (Nj) edge [Ejl] (Nl);
\path (Nk) edge [Ekl] (Nl);
}{}

\end{tikzpicture}
}

\allowdisplaybreaks

\title{The geometry of Gaussian double Markovian distributions}
\author{Tobias Boege, Thomas Kahle, Andreas Kretschmer and Frank Röttger}
\date{\today}

\keywords{Gaussian, conditional independence, normal distribution,
graphical models, model geometry, smoothness}
\makeatletter
\@namedef{subjclassname@2020}{\textup{2020} Mathematics Subject Classification}
\makeatother
\subjclass[2020]{Primary: 62R01; Secondary: 13C70, 13P25, 14P10, 14P25, 62H22}
\begin{document}

\begin{abstract}
  Gaussian double Markovian models consist of covariance matrices
  constrained by a pair of graphs specifying zeros simultaneously in
  the matrix and its inverse. We study the semi-algebraic
  geometry of these models, in particular their dimension, smoothness
  and connectedness as well as algebraic and combinatorial properties.
\end{abstract}

\maketitle

\section{Introduction}
Let $N \coloneqq \set{1, \ldots, n}$ and let $G$ and $H$ be two
undirected, simple graphs on the vertex set $N$. Denote by $\PD_n$ the
set of real symmetric positive definite $n \times n$-matrices. In this
paper, we study the following statistical models.

\begin{definition}\label{def:doubleMarkovModel}
  The \emph{Gaussian double Markovian model of $G$ and $H$} is
  \[
    \M(G,H) \coloneqq \set{ \Sigma \in \PD_n : \left(\Sigma^{-1}\right)_{ij} = 0 \text{ for all } ij \not\in G, \ \Sigma_{kl} = 0 \text{ for all } kl \not\in H }.
  \]
\end{definition}
Denoting the complete graph on $N$ by $K_N$, ordinary undirected
Gaussian graphical models are contained in
\cref{def:doubleMarkovModel} as $\M(G) \coloneqq \M(G,K_N)$.
Covariance models are contained as $\M(K_N,H)$, so both model classes
are unified and generalized here.

\subsection*{Conditional independence and graphical modeling}

Conditional independence (CI) constraints are a central tool in
mathematical modeling of random events.  For random variables
$X_{1},\dots, X_{n}$, basic conditional independence statements
stipulate that one random variable $X_{i}$ is independent of another
$X_{j}$ given a collection of remaining variables $(X_{k})_{k\in K}$
where $i,j \in \set{1,\dots, n}$ and
$K\subseteq \set{1,\dots, n}\setminus \set{i,j}$.  CI~constraints for
discrete and normally distributed random variables translate into
algebraic conditions on elementary probabilities in the discrete case
and on covariance matrices in the Gaussian case.  Algebraic statistics
aims to understand algebraic and geometric properties of conditional
independence models and to relate them to properties of statistical
inference procedures.

In this paper we exclusively treat Gaussian random variables, i.e., we
assume that $(X_{1},\dots, X_{n})$ has a multivariate normal
distribution with positive definite covariance matrix
$\Sigma \in \PD_{n}$.  Since the theory of CI is insensitive to the
mean, we may restrict to centered distributions.

In graphical modeling, edges and paths represent correlation or
interaction and, conversely, notions of disconnectedness represent
independence.  Double Markovian models are conditional independence
models whose CI statements are of the following two forms: either
$\ind{X_{i}}{X_{j}}$ (for each non-edge $ij$ of~$H$) or
$\ind{X_{i}}{X_{j}}[X_{N\setminus \set{i,j}}]$ (for each non-edge $ij$
of~$G$).  Pairwise conditional independence as in the second type is
common in graphical modeling \cite{lauritzen1996graphical}.  The
resulting \emph{Gaussian graphical model} of a simple undirected graph
$G$ on $N$ with edge~set~$E_{G}$~is
\[
\M(G) = \M(G, K_N) = \set{ \Sigma \in \PD_{n} : (\Sigma^{-1})_{ij} = 0
	\text{ for all } ij \notin G}.
\]
In words, the non-edges of $G$ specify zeros of the inverse covariance
matrix $K = \Sigma^{-1}$, also called the \emph{concentration matrix}.
Gaussian graphical models first appeared in \cite{SK86},
and~\cite{uhlerSurvey} is a modern survey containing many connections
to e.g.\ optimization and matrix completion.  Marginal independence
constraints (as in the first case) also appear, for example with
bidirected~\cite{DR2002} or dashed graphs \cite{CW93}.  These models
$ \M(K_N,H) $ encode marginal independence constraints and are also
known as covariance graph models
\cite{kauermann96,lauritzen2020locally}.

A model similar to double Markovian models appeared recently in
\cite[Example~3.4]{lauritzen2020locally}, where
the authors consider graphical models with some entries of $K$ zero
and complementary entries of $\Sigma$ nonnegative.  They investigate
efficient estimation procedures.  These models go back to work of
Kauermann \cite{BR2003,kauermann96}.  One way to describe them
is via mixed parametrizations: the regular exponential family of all
multivariate mean-zero Gaussians can be parametrized by the mean
parameter $\Sigma = (\sigma_{ij})$ or the natural parameter
$K = (k_{st})$.  One can also employ a mixed parametrization, using
$\sigma_{ij}$ and $k_{st}$ for $ij\in A$ and $st\in B$, where
$A\dot\cup B$ is a partition of the entries of an
$n\times n$-symmetric matrix.  Double Markovian models with
$G \cup H = K_{N}$ arise from imposing zeros in the mixed
parametrization.  In general they do not form regular exponential
families, though.  In the terminology of \cite{lauritzen2020locally},
linear constraints on mixed parameters define \emph{mixed linear
	exponential families}.  Such models also appear in causality
theory~\cite{pearl1994can}.

We study geometric properties of statistical models since they can
imply favorable statistical properties.
The asymptotic behavior of $M$-estimators like the MLE depends on
properties of tangent cones that go under the name \emph{Chernoff
  regularity} in~\cite{geyer1994asymptotics}.  Drton has shown that
the nature of singularities determines the large sample asymptotics of
likelihood ratio tests~\cite{drton2009likelihood}.  Smoothness of
log-linear models for discrete random variables has been studied by
examining the parametrization by marginals
\cite{EvansSmoothness,ForcinaSmoothness}.  Generally, smoothness is
favorable because estimation procedures using analytical techniques
like gradient descent rely on it.

In several occasions, for example in
\cref{corollary:coordinate_planes} geometric niceness results follow
because either $\M(G,H)$ or its inverse $\M(H,G)$ is an ordinary
graphical model and thus irreducible, connected, and smooth.  In these
cases one has found an effective new parametrization of $\M(G,H)$.
This theme has occurred in the literature.  For example
\cite{drton08a} asks for a Markov equivalent directed and undirected
graph to the bidirected (i.e.\ covariance) graph.  Our geometric
niceness results, however, go beyond recognizing disguised graphical
models.

A systematic analysis of smoothness of Gaussian CI models has been
initiated in \cite{drton2010smoothness}.
That paper treats the $n=4$ case in detail.  It relies on similar
algebraic techniques as we do here, but also on the characterization
of realizable $4$-gaussoids from~\cite{LM07}.  We deal with a smaller
class of models here, but achieve results independent of the number of
random variables, aiming to understand how the geometry of $\M(G,H)$
depends on $G$ and~$H$.  In particular, we are interested in
dimension, smoothness, irreducible decompositions and other basic
geometric facts that seem useful and interesting for inference
methodology.

Algebraically, a double Markovian model $\M(G,H)$ is the vanishing set
inside $\PD_n$ of an ideal generated by some entries of
$\Sigma \in \PD_{n}$ and some more of its inverse.  The latter is an
algebraic condition as it can be encoded as the vanishing of
submaximal minors of~$\Sigma$.  This puts us broadly in the framework
of sparse determinantal ideals, see \cref{def_CI_ideal} for the
concrete class of ideals we are concerned with.  The sparsity is
twofold in our setting: our ideals are generated by only \emph{some}
minors of a \emph{sparse} generic symmetric matrix, i.e., a symmetric
matrix whose entries in the upper triangle are either distinct
variables or zero.  To our knowledge, no systematic study of these
ideals has been carried out, even in the case of submaximal
minors~only.  Minors of symmetric matrices are a classical topic in
commutative algebra, see for
example~\cite{conca1994symmetric,conca2019}.  Our results, in
particular \cref{theorem:CI_ideal}, can be viewed as a further step
towards the study of this class of sparse determinantal ideals.

\medskip

We illustrate our results on a
simple preliminary example.  \cref{sec:manyExamples} contains further
examples.

\begin{example}\label{example:smooth_non_complete}
	Let $G = \graph4{Eij,Eik,Eil}$ be a star with edge set
	$\set{12, 13, 14}$ and $H = \graph4{Eij,Ejk,Ekl}$ a path with edge
	set $\set{12, 23, 34}$.  To study the model~$\M(G,H)$, consider two
	indeterminate symmetric $4\times 4$-matrices
	$\Sigma = (\sigma_{ij}), K = (k_{ij})$ representing covariance and
	concentration matrices.  The non-edges of $G$ dictate the zeros of
	$K$ and the non-edges of $H$ those of~$\Sigma$.  Algebraically
	(which means ignoring the positive definiteness for a moment), the
	model is specified by the equations
	\begin{equation}\label{eq:exampleEqs}
	\Sigma K = \mathbbm1_{4}, \quad k_{23} = k_{24} = k_{34} =
	\sigma_{13} = \sigma_{14} = \sigma_{24} = 0.
	\end{equation}
	These equations can be solved in \texttt{Macaulay2}~\cite{M2} using
	primary decomposition algorithms.  This computation shows that the
	complex algebraic variety defined by \eqref{eq:exampleEqs} consists
	of two irreducible components.  In one of the components,
	$k_{22} = s_{11} = 0$ holds.  So, if this component contains
	covariance matrices at all, then they are of non-regular Gaussians.
	We do not consider this further in this example, although boundary
	components can be important with respect to marginalization; see
	\cref{ex:SmoothNonminor}.  The other component consists of the matrices
    $\Sigma$ of the~form
	\[
	\Sigma =
	\begin{pmatrix}
	\sigma_{11} & \sigma_{12} & 0 & 0 \\
	\sigma_{12} & \sigma_{22} & 0 & 0 \\
	0 & 0 & \sigma_{33} & 0 \\
	0 & 0 & 0 & \sigma_{44}
	\end{pmatrix},
	\quad
	K =
	\begin{pmatrix}
	k_{11} & k_{12} & 0 & 0 \\
	k_{12} & k_{22} & 0 & 0 \\
	0 & 0 & k_{33} & 0 \\
	0 & 0 & 0 & k_{44}
	\end{pmatrix},
	\]
	subject to the constraints $\Sigma K = \mathbbm1_{4}$.
	This component is 5-dimensional and contains positive definite matrices.
	If one additionally normalizes the variances as $\sigma_{ii} = 1$
	for $i=1,\dots, 4$, then $\Sigma$ is positive definite exactly if
	$\sigma_{12} \in (-1,1)$.
	The (complex) singular locus of the interesting component is empty.
	In fact, smoothness is clear from the simple parametrization.
\end{example}

Several features of this example follow from results in this paper.
For example, the vanishing ideal of the model is a monomial ideal
by~\cref{theorem:CI_ideal}.  That $\M_{1}(G,H)$ is a curve segment is
explained by \cref{proposition:single_edge}. The block-diagonal
structure follows from \cref{theorem:Aida}.

\begin{remark}\label{r:KSigma}
  For computation it is sometimes useful to work in $(\Sigma,K)$-space
  as in Example~\ref{example:smooth_non_complete}.  Let
  $\Sigma = (\sigma_{ij})$ and $K = (k_{st})$ be generic symmetric
  matrices (possibly with ones on the diagonal).  To computationally
  answer algebraic questions about constrained covariance matrices, one
  considers the rings $\RR[\sigma_{ij}: i \le j]$ and
  $\RR[\sigma_{ij}, k_{st}: {i \le j}, {s \leq t}]$ of polynomials with
  real coefficients and whose indeterminates stand for entries of the
  symmetric matrices $\Sigma = (\sigma_{ij})$ and $K = (k_{st})$.  To
  impose equational constraints, one forms quotients by ideals
  generated by the equations.  For example, the relation that
  $\Sigma K = \mathbbm{1}_{n}$ is implemented by construction of the
  quotient ring
  $\RR[\sigma_{ij}, k_{st}: i \le j, s \leq t]/(\Sigma K -
  \mathbbm{1}_n)$.  Another useful trick is to impose non-vanishing or
  invertibility of certain polynomials, for example $\det(\Sigma)$.
  This is achieved by localization.  The~ring
  $\RR[\sigma_{ij}: i \le j]_{\det(\Sigma)}$ is an enlarged version of
  $\RR[\sigma_{ij}: i \le j]$ where now $\det(\Sigma)$ is invertible.
  In fact, the natural map
  \begin{equation*}
    \RR[\sigma_{ij}: i \le j]_{\det(\Sigma)} \overset{\cong}{\longrightarrow}
    \RR[\sigma_{ij}, k_{st}: i \le j, s \leq t]/(\Sigma K - \mathbbm{1}_n),
  \end{equation*}
  defined by mapping $\sigma_{ij}$ to $\sigma_{ij}$ is an isomorphism,
  as it should be because the constraint that $K = \Sigma^{-1}$ makes
  all variables $k_{st}$ functions of the $\sigma_{ij}$.  If
  $I \subseteq \RR[\sigma_{ij}: i \le j]$ is an ideal, then the
  restriction of its extension in the localization at $\det(\Sigma)$
  agrees with its saturation at~$\det(\Sigma)$. This provides a way to
  study conditional independence ideals in the ring
  $\RR[\sigma_{ij}, k_{st}: {i \le j}, {s \leq t}]/\ideal{\Sigma K -
    \mathbbm{1}_n}$, where almost-principal minors of high degree are
  directly available as the $k_{st}$ variables.  Saturation at
  $\det(\Sigma)$, however, is of course not equivalent to the
  saturation at all principal minors.  We recommend
  \cite[Chapter~3]{SullivantBook} for a general introduction to
  computational methods of commutative algebra with a view towards
  statistics.
\end{remark}

\begin{remark}\label{rem:Markovity}
  The term \emph{double Markov property} appears in \cite[Lemma
  1]{kaced2013conditional} based on \cite[Exercise 16.25,
  p.~392]{csiszar2011information} where it is called \emph{double
    Markovity}.  It describes constraints on three random variables
  which are in a special \emph{pair of Markov chains}.  This notion is
  unrelated to the Markovness with respect to two undirected graphs
  studied here.  We judge the potential for confusion low enough to
  reuse this term.
\end{remark}

\subsection*{Overview of results}

The core of our work are several geometric and algebraic insights
having implications for statistical procedures dealing with $\M(G,H)$.
In these results it is often useful to restrict dimension and work
with correlation matrices, which are covariance matrices with ones on
the diagonal.  We write $\PD_{n,1}$ for the set of positive definite
matrices with ones on the diagonal.  It is bounded and known as the
\emph{elliptope}.  Then $\M_1(G,H) \coloneqq \M(G,H) \cap \PD_{n,1}$
is a \emph{correlation model}. Let $ E_G $ denote the edge set of
$ G $ and $ G\cap H $ the graph on $ N $ with edge set
$ E_G\cap E_H $, and similarly $G \cup H$ the graph with edge set
$ E_G\cup E_H $.  An important insight is that geometric and algebraic
properties of double Markovian models often depend on features or the
simplicity of $G\cap H$.  The first result is a decomposition theorem
relying on a notion of direct sum defined via block matrices in
\cref{sec:MinorsEtc}.

\begin{theoremAlph}[\Cref{theorem:Aida,c:trivialModel}]
  Let $V_1, \ldots, V_r$ be a partition of $N$ such that each $V_i$ is
  the vertex set of a connected component of $G \cap H$. Then
  \begin{equation*}
    \M(G,H) = \bigoplus_{i=1}^r \M(G|_{V_i}, H|_{V_i}),
  \end{equation*}
  i.e., every $\Sigma \in \M(G,H)$ has a block-diagonal structure with
  $r$ diagonal blocks having rows and columns indexed by the $V_i$. In
  particular, the correlation model satisfies
  $\M_1(G,H) = \set{\mathbbm{1}_n}$ if and only if
  $E_G \cap E_H = \emptyset$.
\end{theoremAlph}

The next result exhibits that also the union of $G$ and $H$
contributes.  If it is complete, then the imposed constraints are
simple enough to show, for example, smoothness.

\begin{theoremAlph}[\Cref{dim_model,theorem:Piotr,theorem:dim_smooth}]
  For all graphs $G$ and $H$ we have
  $\dim(\M(G,H)) \leq |E_G \cap E_H| + n$. If $G \cup H = K_N$, then
  $\M(G,H)$ is smooth of dimension $|E_G \cap E_H| + n$. Conversely,
  if $\M(G,H)$ attains this maximal dimension, its top-dimensional
  connected components are smooth with irreducible Zariski closure.
\end{theoremAlph}

In \cref{sec:connectedness} we initiate the study of connectedness
of~$\M(G,H)$ in the Euclidean topology.  We conjecture that all double
Markovian correlations models are connected
(\Cref{conj:M1GHconnected}).  This contrasts with the fact that, allowing semi-definite matrices,
similarly defined variants of $\M_1(G,H)$ can consist of isolated
points as in \cref{ex:ZeroDimlModel}.  We have the following results.

\begin{theoremAlph}[\Cref{corollary:coordinate_planes,thm:connected,proposition:single_edge,proposition:classify,proposition:edgeIntersection3}]
  The double Markovian model $\M(G,H)$ is connected in the following
  cases:
  \begin{enumerate}
  \item For every non-edge $kl$ of $G$ there is at most one path $p$
    in $H$ connecting $k$ and~$l$ (or if this holds for $G$ and $H$
    exchanged).
  \item There is a vertex $i \in N$ such that for all non-edges $kl$
    of $G$, every path in $H$ connecting $k$ and $l$ contains~$i$ (or
    if this holds for $G$ and $H$ exchanged).
  \item $|E_G \cap E_H| \leq 3$.
  \end{enumerate}
\end{theoremAlph}

The next theorem is our main algebraic result, see
\cref{theorem:CI_ideal} for a more general version.  Here a forest is
a (not necessarily connected) graph with no cycles.

\begin{theoremAlph}[\Cref{theorem:CI_ideal}]
  Let $G$ be any graph and $H$ a forest.  Then the vanishing ideal of
  $\M_1(G,H)$ is the square-free monomial ideal
  \[
    \mathcal{I}(\M_1(G,H))
    =
    \ideal{\sigma_{ij}, \sigma_p: ij \not\in H, \ p \text{ path in } H \text{ with } e(p) \not\in G}.
  \]
  Here, $\sigma_p$ is the product of variables corresponding to edges
  in $p$, and $e(p)$ denotes its endpoints.
\end{theoremAlph}

Finally, in
\crefrange{proposition:single_edge}{proposition:edgeIntersection3} we
give a classification up to symmetry and matrix inversion of all
double Markovian models with $|E_G \cap E_H| \leq 3$.

\section{Preliminaries on Conditional Independence Structures}
\label{CIIntro}

\subsection{Gaussian conditional independence}

Gaussian graphical models as well as the double Markovian models are
conditional independence models: they are sets of Gaussian distributions
specified by conditional independence assumptions derived from a graph
or pair of graphs. The conditional independence relations of random
variables can be studied combinatorially, using abstract properties
of conditional independence instead of concrete numerical data like a
density function or covariance matrix. To this end, we introduce formal
symbols $(ij|K)$ where $i\neq j \in N$ and $K\subseteq N \setminus ij$.
These formal symbols are subject to the efficient \emph{\Matus set
notation} where union is written as concatenation and singletons are
written without curly braces.  For example, $ijK$ is shorthand for~$\{i\}\cup\{j\}\cup K$.

The symbol $(ij|K)$ shall represent the conditional independence
$\ind{X_{i}}{X_{j}}[X_{K}]$ where $X_{K} = (X_{k})_{k\in K}$.  For
Gaussian random variables $X_{1},\dotsc, X_{n}$, the CI statement
$\ind{X_I}{X_J}[X_K]$ is equivalent to $\rk \Sigma_{IK,JK} = |K|$
by~\cite[Proposition~4.1.9]{SullivantBook}.
Using the adjoint formula for the inverse of a matrix, it can be seen
that a statement $(ij|N\setminus ij)$ (as it appears in the definition
of a graphical model) is equivalent to $(\Sigma^{-1})_{ij} = 0$.

If $I = i$ and $J = j$ are singletons, the rank condition is equivalent
to the vanishing of the determinant of the square submatrix $\Sigma_{iK,jK}$.
These determinants are \emph{almost-principal minors}. It is well-known
that the statements $(ij|K)$ completely describe the entire CI~relation
of a random vector~\cite[Section~2]{MatusMinors}.
The set of all conditional independence statements among $n$ random
variables is
$\A_{n} = \set{ (ij | K) : i\neq j \in N, K\subseteq N\setminus ij}$.
An abstract conditional independence relation is a subset of~$\A_{n}$.
Fundamental problems in the intersection of probability, computer
science and information theory concern the set of \emph{realizable} subsets
$\mathfrak{R} \subseteq 2^{\A_{n}}$, meaning that for
$\C R\in\mathfrak{R}$ there is a random vector $X$ satisfying all
statements in $\C R$ and none of those in~$\A_{n}\setminus \C R$.

To each positive definite $n\times n$-matrix $\Sigma$ we associate a
corresponding CI relation
\[
  \CIS{\Sigma} \defas \set{(ij|K) : \rk \Sigma_{iK,jK} = |K|} \subseteq \A_{n},
\]
consisting exactly of the CI statements satisfied by~$\Sigma$.
Conversely, the covariance matrices satisfying all statements of a
CI~relation $\C R$ form its \emph{Gaussian conditional independence model}:
\begin{equation*}
  \M(\C R) \coloneqq \set{ \Sigma \in \PD_n: \det(\Sigma_{iK,jK}) = 0
    \text{ for all } (ij|K) \in \C R } \subseteq \PD_n
  \subseteq \Sym^2(\RR^n) \cong \RR^{\binom{n+1}{2}}.
\end{equation*}
We consider this set together with the subset topology with respect to
the Euclidean topology on the set of symmetric matrices~$\Sym^2(\RR^n)$.
The cone $\PD_n$ is open in $\Sym^2(\RR^n)$ as it is the preimage of
$\RR_{>0}^n \subseteq \RR^n$ under the continuous map which sends
$\Sigma \in \Sym^2(\RR^n)$ to the vector in $\RR^n$ consisting of the
determinants of all $n$ leading principal minors, using Sylvester's
criterion.  Writing $\Sigma = (\sigma_{ij})$, the associated
\emph{correlation matrix} of $\Sigma$ has
$\frac{\sigma_{ij}}{\sqrt{\sigma_{ii}} \sqrt{\sigma_{jj}}}$ as its
$ij$-entry.  Its~diagonal consists of ones and all non-diagonal
entries lie in $(-1,1)$ as $\det(\Sigma_{ij,ij}) > 0$. Therefore,
the set of correlation matrices $\PD_{n,1}$ is an intersection of
$\PD_{n}$ with an affine linear subspace of $\Sym^2(\RR^n)$.  This
yields a subspace topology and also a canonical smooth structure
on~$\PD_{n,1}$, making it into a smooth submanifold of $\PD_n$ of
codimension~$n$.  It~is often convenient to work with the bounded set
of correlation matrices in the model:
\begin{equation*}
  \M_1(\C R) \coloneqq \set{ \Sigma \in \PD_{n,1}: \det(\Sigma_{iK,jK}) = 0 \text{ for all } (ij|K) \in \C R } = \PD_{n,1} \cap \M(\C R).
\end{equation*}
Many favorable properties transfer between $\M_1(\C R)$
and~$\M (\C R)$, especially if they are of differentiable nature, see
\Cref{l:SmoothnessCorrelation}.  Some care is necessary when
considering algebraic properties such as the number of Zariski
irreducible components.  $\M_{1}(\C R)$ is a linear section of
$\M (\C R)$, so its algebraic properties may differ; see
e.g.~\cref{ex:CI_vs_corr_model}.

There is no finite axiomatic characterization of the set
$\mathfrak{R}$ of realizable CI relations that is valid for all~$n$.
Neither in general~\cite{studeny92:_condit} nor for Gaussians
specifically~\cite{sullivant2009gaussian}.
Closure properties of CI relations often have cryptic names going back
to the search for a finite axiomatization.  In this terminology, for
$\Sigma\in\PD_{n}$, the relation $\CIS{\Sigma}$ is a weakly transitive,
compositional graphoid.  The compound of these properties is also the
definition of \emph{gaussoid} \cite{LM07}. Gaussoids approximate
Gaussian conditional independence in a similar way to matroids
approximating linear independence \cite{Geometry}.

\subsection{Undirected graphical models}

If $G$ is a graph on $N$, then
\[
  \CIS{G} \defas \set{(ij|K) : \text{$K$ separates $i$ and $j$ in $G$}}
\]
denotes the CI \emph{separation statements}, those that follow from
separation in the graph.  We refer to CI~relations of the form $\CIS{G}$
as \emph{Markov relations}. See \cite{lauritzen1996graphical} for all
details on modeling CI by graphs.
Any relation $\CIS{G}$ is realizable, meaning that it equals
$\CIS{\Sigma}$ for some~$\Sigma\in\PD_{n}$ and one can even pick
$\Sigma$ with all positive correlations~\cite[Theorem~4]{Geometry}.
The models realizing $\CIS{G}$ are smooth as they are inverse linear
spaces and thus parametrized by a diffeomorphism (see \cref{p:MarkovianCI}).
Since the CI~relation $\CIS{G}$ of any graph~$G$ is realizable,
it follows that $\CIS{G}$ is a gaussoid. In~addition, graph separation
is upward-stable, meaning that $(ij|L)$ implies $(ij|kL)$ for all
$k \in N \setminus ijL$, and being an upward-stable gaussoid is even a
characterization of being of the form $\CIS{G}$ for some undirected
graph~$G$ by \cite[Proposition~2]{MatusMinors}.

A \emph{pseudographoid} is an abstract CI structure which satisfies
the intersection property, which together with the semigraphoid axiom
forms the definition of \emph{graphoid}; see \cite[Remark~1]{LM07} for
the terminology.
The following lemma states that it is sufficient to verify that
$\Sigma$ satisfies the maximal separation statements (which correspond
to the non-edges in~$G$), for it to satisfy all separation statements
for~$G$.  In this case $\Sigma$ is \emph{Markovian for~$G$}.
\begin{lemma} \label{lemma:Markovian} Let $G$ be an undirected graph
  and $\Sigma$ a complex symmetric matrix with non-vanishing principal
  minors and let $m = \set{(ij|N\setminus ij) : i\neq j}$ denote the
  set of maximal CI statements.  Then
  $\CIS{G} \cap m \subseteq \CIS{\Sigma}$ implies
  $\CIS{G} \subseteq \CIS{\Sigma}$.
\end{lemma}

\begin{proof}
  By~\cite[Corollary~1]{MatusGaussian} the CI~structure $\CIS{\Sigma}$
  is a weakly transitive, compositional graphoid already when $\Sigma$
  is a complex symmetric matrix with non-vanishing principal minors
  (which includes the real positive definite case).
  The lemma follows from \cite[Lemma~3]{LM07}, by choosing
  $\C M = \CIS{G} \cap m$.  Then $G$ is a graph with $i$ and $j$
  adjacent if and only if $(ij|N\setminus ij) \not\in \C M$. Since
  $\C M \subseteq \CIS{\Sigma}$ by assumption, it~follows that
  $\CIS{G} \subseteq \CIS{\Sigma}$.
\end{proof}

By \cref{lemma:Markovian}, $\Sigma$ is Markovian for $G$ if and only
if $\CIS{G} \cap m \subseteq \CIS{\Sigma}$. The maximal CI~statements
$\CIS{G} \cap m$ point out precisely the non-edges in $G$ and therefore
$\Sigma$ being Markovian for $G$ is equivalent to $(\Sigma^{-1})_{ij} = 0$
for all $ij \notin G$. This shows that $\M(G) = \M(\CIS{G})$.

\begin{proposition} \label{p:MarkovianCI}
  Every Markov relation $\CIS{G}$ is realizable by a regular Gaussian
  distribution.  For each graph~$G$, the model $\M(G)$ is irreducible
  and smooth.
\end{proposition}

\begin{proof}
  Realizations for $\CIS{G}$ were constructed from (inverses of)
  generalized adjacency matrices in~\cite[Theorem~1]{LM07}.
  By \cref{lemma:Markovian} the set $\M(G)^{-1} = \set{ \Sigma^{-1} :
  \Sigma \in \M(G) }$ is a linear subspace intersected with the
  cone~$\PD_n$. It is the interior of a \emph{spectrahedron}.
  As such it is an irreducible semi-algebraic set and smooth.
  These properties are transferred to the inverse~$\M(G)$ by
  \cref{l:SmoothnessCorrelation,l:SmoothnessInverse} below.
\end{proof}

Matúš~\cite[Theorem~2]{MatusLogConvex} proved a geometric
characterization of sets $\M(G)^{-1}$: among all Gaussian conditional
independence models, they are precisely those which are convex subsets
of $\PD_n$.

\subsection{Minors, duality and direct sums}
\label{sec:MinorsEtc}

Marginalization and conditioning are natural operations on random
vectors and can also be carried out on conditional independence
structures.  These abstract operations mimic the effect of statistical
operations on a purely formal level.

\begin{definition}\label{def:CIminors}
  Let $\C R \subseteq \A_N$ and $k \in N$. The \emph{marginal}
  and the \emph{conditional} of $\C R$ on $N \setminus k$ are,
  respectively,
\begin{align*}
  \Marg{\C R,k} &\defas \set{ (ij|K) \in \A_{N \setminus k} : (ij|K)  \in \C R }, \\
  \Cond{\C R,k} &\defas \set{ (ij|K) \in \A_{N \setminus k} : (ij|k\cup K) \in \C R }.
\end{align*}
Any set $\C R' \subseteq \A_{N'}$, $N' \subseteq N$, obtained from
$\C R$ by a sequence of marginalization and conditioning operations is
a \emph{minor} of~$\C R$.
\end{definition}
On covariance matrices, marginalizing away a variable $k \in N$ is
achieved by taking the principal submatrix $\Marg{\Sigma,k} \defas
\Sigma_{N \setminus k}$. The conditional distribution on~$k$ is the
Schur complement of the $k \times k$ entry $\Cond{\Sigma,k} \defas
\Sigma_{N \setminus k} - \sigma_{kk}^{-1} \Sigma_{N\setminus k,k} \cdot
\Sigma_{k,N\setminus k}$. This is proven in \cite[Theorem~2.4.2]{SullivantBook}.

\begin{lemma}
For $\Sigma \in \PD_n$ we have
$\Marg{\CIS{\Sigma},k} = \CIS{\Marg{\Sigma,k}}$ and
$\Cond{\CIS{\Sigma},k} = \CIS{\Cond{\Sigma,k}}$.
In particular, minors of realizable CI~relations are realizable.
\end{lemma}

For any Gaussian distribution, matrix inversion exchanges the
covariance and concentration matrices.  The combinatorial version of
this operation furnishes an involution on CI~relations.

\begin{definition}\label{def:CIdual}
  The \emph{dual} of $\C R \subseteq \A_N$ is
  $ \C R^\dual \defas \set{ (ij|N \setminus ijK) : (ij|K) \in \C R }
  \subseteq \A_N.  $
\end{definition}

This involution turns a covariance matrix $\Sigma$ which is Markovian
for a graph~$G$ into a concentration matrix $K = \Sigma^{-1}$ such
that $K_{ij} = 0$ for all $ij \notin G$. It also exchanges marginal
and conditional~\cite[Lemma~1]{LM07}:

\begin{lemma}\label{lem:dualInverse}
For any $\C R \subseteq \A_N$ and $k \in N$ we have
$\Marg{\C R^\dual,k} = (\Cond{\C R,k})^\dual$.
If $\Sigma$ is positive definite, then $\CIS{\Sigma}^\dual = \CIS{\Sigma^{-1}}$.
In particular, duals of realizable CI~relations are realizable.
\end{lemma}

The final operation of interest is concatenating two independent
Gaussian random vectors $(X_i)_{i \in N}$ and $(Y_i)_{i \in M}$ which
are indexed by disjoint ground sets $N$ and $M$. This is called
\emph{direct sum} in the structure theory of CI~relations
\cite{matus94}.  The corresponding CI relation is as follows.

\begin{definition}
  Let $\C R$ and $\C R'$ be two CI structures on disjoint ground sets
  $N$~and~$M$, respectively.  Their \emph{direct sum} is the
  CI~structure
\begin{alignat*}{2}
  \C R \oplus \C R' :=  &\; \set{ (ij|K) \in \A_{NM} : i \in N, j \in M } \\
    \cup&\; \set{ (ij|KL) \in \A_{NM} : (ij|K) \in \C R, L \subseteq M } \\
    \cup&\; \set{ (ij|KL) \in \A_{NM} : (ij|K) \in \C R', L \subseteq N }
    \subseteq \A_{NM}.
\end{alignat*}
\end{definition}

On the level of covariance matrices, the direct sum imposes a
block-diagonal structure with the summands on the diagonal. For
$\Sigma\in\PD_{N}$ and $\Sigma' \in \PD_{M}$ let
$\Sigma \oplus \Sigma' =
\begin{psmallmatrix}
  \Sigma & 0 \\
  0 & \Sigma'
\end{psmallmatrix}
$.  The following lemma is immediate.
\begin{lemma} \label{l:DirectSum}
For $\Sigma \in \PD_N$ and $\Sigma' \in \PD_M$ we have
$\CIS{\Sigma \oplus \Sigma'} = \CIS{\Sigma} \oplus \CIS{\Sigma'}$.
In particular, the direct sum $\C R \oplus \C R'$ is realizable
if and only if $\C R$ and $\C R'$ are both realizable.
Moreover, the direct sum commutes with duality and minors.
\end{lemma}

\begin{lemma} \label{SmoothnessDirectSum}
For $V = \M(\C R)$ and $V' = \M(\C R')$ on disjoint ground sets
$N$ and $M$, respectively, the direct sum $U = V \oplus V'$ on
$NM$ is smooth if and only if $V$ and $V'$ are both smooth.
\end{lemma}

\begin{proof}
  From the block-diagonal shape of matrices in $U$ and
  \cref{l:DirectSum} it follows that:
\begin{enumerate}
\item $V$ and $V'$ are irreducible if and only if $U$ is
  irreducible,

\item $T_{\Sigma \oplus \Sigma'} U = T_{\Sigma} V \oplus T_{\Sigma'} V'$,
  and

\item $\dim U = \dim V + \dim V'$.
\end{enumerate}
Given irreducibility, smoothness means equality of the tangent space
dimension to the model dimension.  Therefore the smoothness conditions
are equivalent.
\end{proof}

\begin{remark}
  Any direct summand of a CI~relation is a marginalization.
  Marginalizations in general need not preserve smoothness, as
  \cref{ex:SmoothNonminor} below shows.  But, one direction of
  \cref{SmoothnessDirectSum} yields that a direct summand of a smooth
  model is smooth.  Consequently, the non-zero entries in off-diagonal
  blocks are obstructions to smoothness of marginalizations.
\end{remark}

The corresponding operations on graphs have been explained
in~\cite{MatusMinors}. For a graph $G$ and a vertex~$k$, write
$\Marg{G,k}$ for the graph~$G$ where vertex~$k$ and all incident
edges are deleted and $\Cond{G,k}$ for the graph~$G$ where
vertex~$k$ is deleted and all vertices previously adjacent
to~$k$ are connected to form a clique. The direct sum $G \oplus G'$
of graphs $G, G'$ on disjoint ground sets $N, M$, respectively,
consists of the disjoint unions of the vertex and edge sets of
$G$ and $G'$ forming an undirected graph on~$NM$.
The operations on CI~relations, positive definite matrices and
graphs are all aligned:

\begin{lemma}\label{l:graph-minor}
Let $G$ be a graph on vertex set $N$ and $k \in N$. We have
$\Marg{\CIS{G},k} = \CIS{\Cond{G,k}}$,
$\Cond{\CIS{G},k} = \CIS{\Marg{G,k}}$ and
$\CIS{G \oplus G'} = \CIS{G} \oplus \CIS{G'}$.
In particular, Markov relations are closed under forming minors
and direct sums.
\end{lemma}
Duality has no counterpart in undirected graphical models. If undirected
graphical models are referred to as ``concentration models'', their duals
are ``covariance models''. Sometimes they are written with bidirected
edges instead of undirected ones.

\subsection{Double Markovian models}

In a double Markovian relation a pair of graphs $(G,H)$ specifies vanishing
conditions via $G$ on the concentration matrix and via $H$ on the covariance
matrix. This can be expressed using duality on ordinary Markov relations:

\begin{definition}
Let $G$ and $H$ be undirected graphs on vertex set~$N$. Their \emph{double
Markov relation} is $\CIS{G,H} = \CIS{G} \cup \CIS{H}^\dual$.
\end{definition}

By \Cref{lemma:Markovian}, $\M(G,H) = \M(\CIS{G,H})$ and the following
is similar to \cref{l:graph-minor}.

\begin{lemma} \label{l:DM-minorclosed}
  Let $G$ and $H$ be graphs on the vertex set~$N$. Then
  \begin{equation*}
    \CIS{G,H}^\dual = \CIS{H,G}, \quad
    \Marg{\CIS{G,H},k} = \CIS{\Cond{G,k},\Marg{H,k}}, \quad
    \Cond{\CIS{G,H},k} = \CIS{\Marg{G,k},\Cond{H,k}}.
  \end{equation*}
  Hence the class of double Markov relations is minor-closed.
\end{lemma}

All Markov relations are realizable and their models are smooth.
These desirable geometric properties fail to hold for double Markov
models. The CI~relation $\CIS{G,H}$ gives only partial information
about the geometry of the model. It is incomplete in the sense that
$\CIS{G,H}$ does not contain all CI~statements which are true on the
model~$\M(G,H)$. The remainder of this section contains examples of
pathological behavior in the double Markovian setting.

Our first example of this is a double Markov relation which is not
even a semigraphoid because it violates the following basic rule,
which is satisfied by any probability distribution (Gaussian or not)
and therefore must hold on the model $\M(G,H)$:
\[
  (12|) \land (13|2) \; \Rightarrow \; (13|) \land (12|3).
\]

\begin{example}
  \label{ex:veryNotRealizable}
  Consider the two graphs $G = \graph4{Eij,Ejk}$ and
  $H = \graph4{Eik}$, where the vertices are labeled $1,2,3,4$,
  clockwise starting at the top left.  The vertex $4$ can be replaced
  by any graph as long as it is not connected to $123$.  Because
  $(13|) \not\in \CIS{G}$ and $(13|24) \not\in \CIS{H}$,
  $(13|) \not\in \CIS{G,H}$.  However, $(13|2)$ holds in $G$ and
  $(12|L)$ holds in~$H$ for any $L\subseteq 34$.  But $(12|)$ and
  $(13|2)$ without $(13|)$ in $\CIS{G,H}$ contradict the semigraphoid
  property, because by contraction
  $(12|) \land (13|2) \Rightarrow (1,23|)$ and then by decomposition
  $(1,23|) \Rightarrow (13|)$ and by weak union
  $(1,23|) \Rightarrow (12|3)$.
\end{example}

The semigraphoid closure of the CI structure in
\cref{ex:veryNotRealizable} (with $4$ being an isolated vertex)
is~$\A_4$ and its model consists only of the identity matrix.  As
explained by \Cref{c:trivialModel}, this happens because the edge sets
of $G$ and $H$ are disjoint.

The next example uses the weak transitivity axiom
$(12|) \land (12|3) \Rightarrow (13|) \lor (23|)$
to construct a family of graph pairs $(G,H)$ whose model is non-smooth.
The basic idea is that the logical OR in the conclusion of weak transitivity
produces two components of the model.  We verify that the CI model of
$\set{(12|), (12|3)}$ has two irreducible components which intersect
in the positive definite cone.  In particular, this model is not
smooth.

\begin{example}
  \label{ex:Nonsmooth} Pick
  $G = H = \graph4{Eik,Ejk}$ which connects $1$ with $2$ via $3$.
  Making only small changes to the following arguments, the fourth
  node can be replaced by any (possibly empty) graph that is not
  connected to $123$.  The CI structure is
  \begin{gather*}
    \CIS{G,H} = (\A_4
      \setminus \set{ (13|L) : L \subseteq 24 })
      \setminus \set{ (23|L) : L \subseteq 14 }.
  \end{gather*}
  In particular the formula
  $(12|) \land (12|3) \land \lnot(13|) \land \lnot(23|)$ holds for
  $\C R = \CIS{G,H}$, which violates weak transitivity.
  Because it violates weak transitivity, $\C R$ is not a gaussoid and
  not realizable by a regular Gaussian distribution.  There are two
  gaussoid extensions of $\C R$ to consider: $\C R_1$~which adds $(13|L)$
  and $\C R_2$ which adds $(23|L)$, for all $L$, respectively, to $\C R$.
  These extensions are isomorphic by exchanging the roles of $1$~and~$2$.
  They are Markov relations corresponding to the complete graph with one
  edge removed. Hence, they are realizable and their models are irreducible
  and smooth by \Cref{p:MarkovianCI}.
  However, the model of $\C R$ consists of two copies of this smooth model,
  intersecting at the identity matrix, which makes it a singular~point.
\end{example}

Unlike being double Markovian, being double Markovian with a smooth
model is not minor-closed. Smoothness fails because marginalizations
of irreducible models can be reducible.

\begin{example} \label{ex:SmoothNonminor}
  The two graphs $G = \graph4{Eil,Ejk,Ekl}$ and $H = \graph4{Eik,Ejk}$
  impose the following relations on a positive definite $4 \times 4$ matrix
  $\Sigma = (\sigma_{ij})$ in their double Markovian model:
\begin{gather*}
  \sigma_{12} = \sigma_{14} = \sigma_{24} = \sigma_{34} = 0, \; \text{from $\CIS{H}^\dual$}, \\
  \sigma_{13} \sigma_{23} \sigma_{44} = 0, \; \sigma_{13} \sigma_{22} \sigma_{44} = 0, \;
    \text{from $\CIS{G}$}.
\end{gather*}
The bounded model $\M_1(G,H) = \M(G,H) \cap \PD_{4,1}$ is a curve segment
parametrized by $\sigma_{23} \in (-1,1)$ since $\sigma_{13}$ is forced to
zero on~$\PD_{4,1}$.
The marginal CI~structure on $123$ is $\Marg{\CIS{G,H},4} =
\CIS{\Cond{G,4},\Marg{H,4}} = \CIS{\graph3{Eik,Ejk},\graph3{Eik,Ejk}}$,
the one from \cref{ex:Nonsmooth}. Its model has \emph{two} components
which intersect in the identity matrix and is therefore not smooth.

To understand this phenomenon one has to distinguish the model of the
marginal CI~structure, $\M(\Cond{G,4},\Marg{H,4})$, from the pointwise
marginalization of~$\M(G,H)$.  What is discussed above is the former.
It is reducible and properly contains the latter model as one of its
two components.

The ``unexpected'' component of $\M(\Cond{G,4},\Marg{H,4})$ arises
from semi-definite matrices on the boundary of $\M(G,H)$ which become
regular after marginalization.  Namely, the two equations
\[
  \sigma_{13} \sigma_{23} \sigma_{44} = 0, \; \sigma_{13} \sigma_{22} \sigma_{44} = 0
\]
imply $\sigma_{13} = 0$ on positive definite matrices, but there are
semi-definite solutions to them where
\begin{inparaenum}
\item $\sigma_{22} = 0$ and $\sigma_{23} = 0$, or
\item $\sigma_{44} = 0$, and $\sigma_{13}$ and $\sigma_{23}$ are arbitrary.
\end{inparaenum}
Thus there are three types of solutions:
\[
  \begin{pmatrix}
    \sigma_{11} & 0 & 0 & 0 \\
    0 & \sigma_{22} & \sigma_{23} & 0 \\
    0 & \sigma_{23} & \sigma_{33} & 0 \\
    0 & 0 & 0 & \sigma_{44}
  \end{pmatrix},
  \qquad
  \begin{pmatrix}
    \sigma_{11} & 0 & \sigma_{13} & 0 \\
    0 & 0 & 0 & 0 \\
    \sigma_{13} & 0 & \sigma_{33} & 0 \\
    0 & 0 & 0 & \sigma_{44}
  \end{pmatrix},
  \qquad
  \begin{pmatrix}
    \sigma_{11} & 0 & \sigma_{13} & 0 \\
    0 & \sigma_{22} & \sigma_{23} & 0 \\
    \sigma_{13} & \sigma_{23} & \sigma_{33} & 0 \\
    0 & 0 & 0 & 0
  \end{pmatrix}.
\]
The first type is visible in $\PD_4$ and in $\M(G,H)$ and its
marginalization forms one of the components of
$\M(\Cond{G,4},\Marg{H,4})$.  The second type of solutions is not
visible in the marginalization because it contains no $\PD_{3}$
matrices.  Marginalizing~$4$ from the CI structure removes the last
row and column of these matrices and imposes additional constraints,
in this case $\sigma_{13}\sigma_{23} = 0$.  This turns the third type
of solution positive definite and reduces the dimension by one.  In
this case, both the first and second component arise from this
2-dimensional boundary component.
\end{example}

Next, we present a family of double Markov relations which are
realizable but whose model is singular at the identity matrix.
Moreover, this gives an infinite family of realizable, non-smooth
models all of whose proper minors are realizable and smooth.

\begin{example}
  \label{ex:SelfdualNonsmooth}
  By~\cite[Proposition~4.2]{drton2010smoothness}, the model of
  $\C R = \set{ (12|), (12|N \setminus 12) }$ is singular at the
  identity matrix.  It~can be represented by the double Markov relation~$\CIS{G,G}$
  where $G$ is the graph on $N$ whose only non-edge is~$12$.
  If~$k=1,2$, the marginalization and conditioning of $G$ by~$k$ are
  both the complete graph, so their double Markovian model is the
  entire cone $\PD_n$ and smooth.  If~$k \in N \setminus 12$, then
  $\Cond{G,k}$ is a complete graph and $\Marg{G,k}$ is another graph
  on $N \setminus k$ whose only non-edge is~$12$.  Since a complete
  graph imposes no relations, we find that $\Marg{\CIS{G,G},k} =
  \CIS{\Marg{G,k}}$ is a Markov relation and thus its model is smooth
  by \cref{p:MarkovianCI}. It is clear that all further minors of
  these two cases yield smooth models as well. This gives infinitely
  many examples of minor-minimal, non-smooth double Markovian
  CI~structures.
  Explicit~computation in \texttt{Macaulay2} for $n=4,5,6$ verifies
  that (the~Zariski closure of) the model is an irreducible variety
  and this shows that singularities are not always caused by an
  intersection of irreducible components.
  In~\cref{ex:pathConnSingular} we examine the singular locus
  of $\M(G,G)$ further.  It~turns out to be another conditional
  independence model.
\end{example}

We have no general proof for the Zariski-irreducibility of the models
in the preceding example when $n \ge 7$. It was shown in
\cite[Lemma~6.4]{Gaussant} that the CI~structures
$\set{(12|), (12|N \setminus 12)}$ are realizable by Gaussian
distributions for all~$n \ge 4$.  Thus the models are irreducible in
the finite ``lattice of CI relations'' studied by Drton and Xiao in
\cite[Section~2.1]{drton2010smoothness}.  This is a coarsening of the
lattice of the usual Zariski topology induced on~$\PD_n$, so
CI-irreducibility provides a necessary condition for
Zariski-irreducibility.

There is a Galois~connection between CI~structures
$\C R \subseteq \A_N$ and Gaussian
CI~models~$\M(\C R) \subseteq \PD_n$.  The closed CI~structures under
this connection are termed \emph{complete relations} by Drton and
Xiao, and their \cite[Theorem~2.2]{drton2010smoothness} characterizes
them as the intersections of realizable gaussoids.
The completion of a double Markovian relation $\CIS{G,H}$ adds all
CI~statements which hold on every matrix in its model
$\M(G,H)$. For~single Markovian relations $\CIS{G}$ is always complete
because it is realizable by \Cref{p:MarkovianCI} and its elements can
be read off from~$G$.

On the other hand, \Cref{ex:veryNotRealizable} exhibits a pair of
graphs $G$ and $H$ whose CI~structure does not satisfy the
semigraphoid property.  Since the semigraphoid axiom holds for every
Gaussian distribution, it follows that $\CIS{G,H}$ need not be
complete. Moreover, even if the set $\CIS{G,H}$ is closed under the
compositional graphoid axioms, which hold for all Gaussians (called
\emph{semigaussoid axioms} by Drton and Xiao), it may still be
incomplete:

\begin{example} \label{ex:Incomplete}
  Let $G = \graph4{Eij,Eil,Ejk,Ekl}$ and $H = \graph4{Eik,Eil,Ejk,Ejl}$ both
  be $4$-cycles. We have $\CIS{G} = \set{ (13|24), (24|13) }$ and $\CIS{H}^\dual
  = \set{ (12|), (34|) }$. Their union $\C R$ is the set of antecedents to
  an instance of inference rule \cite[Lemma~10~(17)]{LM07}:
  \[
    (12|) \wedge (34|) \wedge (13|24) \wedge (24|13) \Rightarrow (13|)
  \]
  This formula is valid for all regular Gaussians. Since $\C R$ is not closed
  under this rule, it is not realizable by a positive definite matrix.
  It~is not complete either because the inference rule~(17) is a \emph{Horn clause},
  i.e., it has a unique consequence~$(13|)$ which every realizable
  superset of $\C R$ and hence $\C R$ as their intersection would have
  to contain if it were complete.

  However, $\C R$ is a gaussoid (therefore closed under the
  compositional graphoid axioms) and it is realizable by a complex
  matrix with non-vanishing principal minors. Consequently, working
  with equations in a computer algebra system like \texttt{Macaulay2}, one
  cannot deduce any further CI~statements from $\CIS{G,H}$.  Positive
  definiteness has to be taken into account.

  The completion of $\C R$ in the positive definite setting can be
  computed as its closure under the semigaussoid axioms
  \cite[Definition~1~(7)--(9)]{LM07} and the higher inference rules
  \cite[Lemma~10 (17)--(21)]{LM07} of Lněnička and Matúš. It equals
  $\ol{\C R} = \A_4 \setminus \left(\set{ (14|L) : L \subseteq 23 }
    \cup \set{ (23|L) : L \subseteq 14 }\right)$.  This structure is a
  self-dual Markov relation and hence can be written as the relation
  of two identical graphs $(J,J)$ such that $\CIS{J} =
  \CIS{J}^\dual$. Indeed $J = G \cap H$ gives $\CIS{J,J} = \ol{\C R}$.
  This~shows that nevertheless
  $\M(G,H) = \M(\ol{\C R}) = \M(J, J) = \M(G\cap H)$ is smooth.
\end{example}

\begin{qu} \label{qu:CIcomplete}
Is there a combinatorial criterion similar to separation in undirected
graphs to derive a complete set of valid CI~statements for $\M(G,H)$?
\end{qu}

A first step is \Cref{c:trivialModel} where the triviality of the
model is characterized by disjointness of the edge sets. In this case,
every CI~statement is a consequence of the statements in~$\CIS{G,H}$.

\section{Geometry of the models \texorpdfstring{$ \M(G,H) $}{M(G,H)}}
\label{sec:geometry}

Our study of the smoothness of double Markovian models starts with the
known observation that one may as well work with the bounded set of
correlation matrices.
\begin{lemma}[{\cite[Lemma~3.2]{drton2010smoothness}}] \label{l:SmoothnessCorrelation}
  The set $\M(\C R)$ is a smooth submanifold of $\PD_n$ if and
  only if $\M_1(\C R)$ is a smooth submanifold.
\end{lemma}

\begin{proof}
  The map $\PD_n \rightarrow \RR_{>0}^n \times \PD_{n,1}$, sending
  $\Sigma$ to the pair consisting of its diagonal and its associated
  correlation matrix, is a diffeomorphism. Consider the following
  commutative diagram:
  \begin{equation*}
    \begin{tikzcd}
      \PD_n \arrow[r, "\cong"]
      & \RR_{>0}^n \times \PD_{n,1}
      & \set{(1,1, \ldots, 1)} \times \PD_{n,1} \arrow[l, hook]
      & \PD_{n,1} \arrow[l, equal] \\
      \M(\C R) \arrow[u, hook] \arrow[r, "\cong"']
      & \RR_{>0}^n \times \M_1(\C R) \arrow[u, hook]
      & \set{(1,1, \ldots, 1)} \times \M_1(\C R) \arrow[l, hook] \arrow[u, hook]
      & \M_1(\C R) \arrow[l, equal] \arrow[u, hook].
    \end{tikzcd}
  \end{equation*}
  All maps are topological embeddings or homeomorphisms, and the upper
  row consists of diffeomorphisms and embeddings of smooth
  submanifolds. If $\M(\C R)$ is a smooth submanifold,
  $\RR_{>0}^n \times \M_1(\C R)$ inherits a smooth manifold structure
  making the second vertical inclusion an embedding of a smooth
  submanifold.  Now, the product
  $\set{(1,1, \ldots, 1)} \times \M_1(\C R)$ is the preimage of a
  regular value under the projection map
  $\RR_{>0}^n \times \M_1(\C R) \rightarrow \RR_{>0}^n$. This again
  can be seen from the composition
  \begin{equation*}
    \M(\C R) \hookrightarrow \PD_n \overset{\cong}{\rightarrow} \RR_{>0}^n \times \PD_{n,1} \rightarrow \RR_{>0}^n,
  \end{equation*}
  sending $\Sigma \in \M(\C R)$ to its diagonal. The claim follows
  because for each $\Sigma \in \M(\C R)$ having only ones on the
  diagonal and an arbitrary positive diagonal matrix $D$, we have
  $D \Sigma D \in \M(\C R)$ (by the proof of
  \cite[Lemma~3.1]{drton2010smoothness}).  Choosing an appropriate
  smooth path of diagonal matrices passing through the identity
  matrix, we obtain every possible tangent vector in
  $T_{(1,1, \ldots, 1)}\RR_{>0}^n$.
	
  If $\M_1(\C R)$ is a smooth submanifold of $\PD_n$, then also
  automatically of $\PD_{n,1}$, and the product
  $\RR_{>0}^n \times \M_1(\C R)$ inherits a canonical smooth structure
  making the second vertical inclusion an embedding of a smooth
  submanifold. Then clearly the induced smooth structure on $\M(\C R)$
  makes the leftmost vertical inclusion an embedding of a smooth
  submanifold as well.
\end{proof}

The proof shows that if $\M_1(\C R)$ is a smooth submanifold of $\PD_n$, then
it is in fact a smooth submanifold of $\M(\C R)$ via the
inclusion. The following lemma is easily verified.

\begin{lemma}[{\cite[Lemma~3.3]{drton2010smoothness}}] \label{l:SmoothnessInverse}
  There is a self-inverse diffeomorphism
  $\PD_{n,1} \rightarrow \PD_{n,1}$, given by matrix inversion
  followed by forming the correlation matrix, mapping $\M_1(\C R)$
  onto $\M_1(\C R^\dual)$.
\end{lemma}

In particular, $\M(\C R)$ and $\M_1(\C R)$ are smooth if and only if
$\M(\C R^\dual)$ and $\M_1(\C R^\dual)$ are.  As an image of a linear
space under the inversion diffeomorphism any graphical model $\M(G)$
is smooth, which might not be obvious from the defining
$(n-1) \times (n-1)$ almost-principal minors.
The (semi-)algebraic geometry of $\M(\C R)$ and $\M_1(\C R)$ can be
quite different in the non-smooth case. E.g., the number of
irreducible components of their Zariski closures need not always
agree. However, the bijective morphism of semi-algebraic sets
$\RR^n_{>0} \times \M_1(\mathcal{R}) \rightarrow \M(\mathcal{R}), (D,
\Sigma) \mapsto D \Sigma D$ can be used to show that their dimensions
always differ by $n$.

\subsection{Basics from real algebraic geometry}
We collect several foundational definitions and results from
\cite{bochnak2013real} and refer to this textbook for an
extensive treatment.

A \emph{real algebraic set} $Z \subseteq \RR^n$ is the
vanishing set $\mathcal{V}(S)$ of a collection
$S \subseteq \RR[x_1, \ldots, x_n]$ of polynomials, and $S$ may
be replaced by the ideal $\ideal{S}$ it generates.  Real algebraic
sets are the closed sets of the \emph{Zariski topology}
on~$\RR^n$.  If $\Theta \subseteq \RR^n$ is any subset,
its ideal is
$\C I(\Theta) = \set{f\in \RR[x_{1},\dotsc,x_{n}] : f(x) = 0,
  \text{ for all } x\in \Theta}$.  The real algebraic set of
$\C I(\Theta)$ is the \emph{Zariski closure $\ol \Theta$ of $\Theta$}.
Every irreducible component of the Zariski closure of $\Theta$ in
$\RR^n$ intersects~$\Theta$.  If~$Z$ is irreducible in this
topology, $Z$ is a \emph{real algebraic variety}.  A set of the form
\begin{equation*}
  \Theta = \set{x \in \RR^n: f_1(x) = \cdots = f_r(x) = 0, g_1(x) > 0, \ldots, g_s(x)> 0},
\end{equation*}
where $f_i, g_j \in \RR[x_1, \ldots, x_n]$ are real polynomials
is a \emph{basic semi-algebraic set}.  A finite union of basic
semi-algebraic sets in a fixed $\RR^n$ is a
\emph{semi-algebraic set}.
The \emph{dimension} of a semi-algebraic set $\Theta \subseteq \RR^n$
is the Krull dimension of its coordinate ring
$\RR[x_1, \ldots, x_n] / \mathcal{I}(\Theta)$.  The dimension of
$\Theta$ then equals the dimension of its Zariski closure as
$\mathcal{I}(\Theta) = \mathcal{I}(\overline{\Theta})$.  A
semi-algebraic set $\Theta \subseteq \RR^n$ is
\emph{semi-algebraically connected} if for every two \emph{disjoint}
semi-algebraic subsets $A, B \subseteq \Theta$ which are closed in
$\Theta$ and satisfy $A \cup B = \Theta$, one has $A = \Theta$ or
$B = \Theta$.  According to \cite[Theorem~2.4.5]{bochnak2013real}, a
semi-algebraic set $\Theta \subseteq \RR^n$ is semi-algebraically
connected if and only if it is connected with respect to the Euclidean
topology on~$\RR^n$.

\begin{definition}\label{def:smooth}
  For a real algebraic set $V \subseteq \RR^n$ with vanishing ideal
  $\mathcal{I}(V) = \ideal{f_1, \ldots, f_r}$, the \emph{Zariski
    tangent space $T_p V$ at $p \in V$} is the kernel of the Jacobian
  matrix
  \begin{equation*}
    J_p = \left( \frac{\partial f_i}{\partial x_j}(p) \right)_{
      \substack{i = 1, \ldots, r, \\ j = 1, \ldots, n}
    }.
  \end{equation*}
  The algebraic set $V$ is \emph{smooth at $p \in V$} if it is
  irreducible and $\dim(T_p V) = \dim(V)$ or, equivalently,
  $\rk(J_p) = n - \dim(V)$.  Finally, $V$ is \emph{smooth}
  if it is smooth at every point $p \in V$.
\end{definition}

Assume $V$ is an irreducible real algebraic set.  Krull's principal
ideal theorem implies that the rank of the Jacobian is at most
$n - \dim(V)$ or equivalently $\dim(T_p V) \geq \dim(V)$.  Moreover,
the rank does not depend on the set of generators of~$\mathcal{I}(V)$.
This follows from the fact that adding an element of the ideal
$\mathcal{I}(V)$ to a list of generators does not change the rank as
the gradient of the additional polynomial is linearly dependent on the
gradients of the generators at every point $p \in V$. This reasoning
also proves the following:
\begin{lemma}\label{lemma:rkInequality}
  Let $V \subseteq \RR^n$ be a real algebraic set with generators
  $f_1, \ldots, f_r$ of $\mathcal{I}(V)$ and let
  $g_1, \ldots g_s \in \mathcal{I}(V)$ be arbitrary.
  Then for all $p \in V$ we have
  \begin{equation*}
    \rk\left( \frac{\partial f_i}{\partial x_j}(p) \right)_{\substack{i = 1, \ldots, r, \\ j = 1, \ldots, n}}
    \geq \rk\left( \frac{\partial g_i}{\partial x_j}(p) \right)_{\substack{i = 1, \ldots, s, \\ j = 1, \ldots, n}}.
  \end{equation*}
\end{lemma}
The natural application of \cref{lemma:rkInequality} is to bound the
$\dim(V)$ when $\mathcal I(V)$ is not known.

A semi-algebraic set $\Theta \subseteq \RR^n$ is \emph{smooth at
  $p \in \Theta$} if $p$ is contained in a unique irreducible
component $Z$ of the Zariski closure of $\Theta$ and $p$ is a smooth
point of~$Z$. If $\Theta$ is smooth at every point, then $\Theta$ is
\emph{smooth}.  The set of smooth points of $\Theta$ is its
\emph{smooth locus}, denoted~$\Theta_{\rm sm}$.  The smooth locus of a
non-empty real algebraic set $V$ is a non-empty Zariski-open subset
of~$V$ by~\cite[Proposition~3.3.14]{bochnak2013real}.
If $V$ is irreducible, $V_{\rm sm}$ is a Zariski-dense open subset.
The smooth locus $V_{\rm sm}$ of a real algebraic variety $V$ is a
smooth submanifold of $\RR^n$ by \cite[Proposition~3.3.11]{bochnak2013real}.
From this it follows that, if $\Theta \subseteq \mathbb{R}^n$ is a
basic semi-algebraic set such that its Zariski closure is irreducible,
then $\Theta_{\rm sm}$ is a smooth submanifold of $\RR^n$ by viewing
$\Theta_{\rm sm}$ as an open subset (in~the Euclidean topology)
of~$\overline{\Theta_{{\rm sm}}}$.

\subsection{Smoothness of the models \texorpdfstring{$\M(G,H)$}{M(G,H)}}
\label{sec:smooth}

The geometry of double Markovian models is best understood in terms of
semi-algebraic sets.  We can identify
\begin{equation*}
  \M(G,H) = \M(G) \cap \M(H)^{-1}
\end{equation*}
with
\begin{equation*}
	\widetilde{\M(G,H)} \coloneqq \set{(\Sigma,\Sigma^{-1}): \Sigma_{ij} = 0 \text{ for all } ij \notin H \text{ and } (\Sigma^{-1})_{kl} = 0 \text{ for all } kl \notin G},
\end{equation*}
and the latter is a smooth submanifold of $\PD_n \times \PD_n$ if and only if $\M(G,H)$ is a smooth submanifold of $\PD_n$, as follows from the diagram:
\begin{equation*}
\begin{tikzcd}
\PD_n \ar[r, hook, "\id \times \inv"]
& \PD_n \times \PD_n \\
\M(G,H) \ar[u, hook] \ar[r, "\cong"]
& \widetilde{\M(G,H)} \ar[u, hook].
\end{tikzcd}
\end{equation*}
Because of this, it suffices to study the smoothness of~$\M(G,H)$.

Both $\mathcal{M}_G$ and $\mathcal{M}_H^{-1}$ are smooth submanifolds
of~$\PD_n$.  It is sufficient for the intersection to be smooth that
the intersection be transverse at every point
\cite[Chapter~1~\S5]{guillemin2010differential}, meaning that the
dimensions of tangent spaces add up to that of the ambient manifold.
This criterion yields \cref{theorem:Piotr}, once we have computed the
tangent spaces.

\begin{proposition}\label{p:tangentSpaces}
  A basis of the tangent space~$T_P \M(G)$ is given by the matrices
  \[M^{ij} \coloneqq P^i \cdot P_j + P^j \cdot P_i, \qquad \text {for
      $i=j$ or $ij \in E_G$},\]
  where $P^i$ is the $i$-th column of $P$ and $P_j$ is the $j$-th row
  of~$P$.
  A basis of the tangent space $T_P \M(H)^{-1}$ consists of
  $E^{ij} \coloneqq E_{ij} + E_{ji}$ for $i=j$ or $ij \in E_H$, where
  $E_{ij}$ is the $n \times n$ matrix having a $1$ at the $(i,j)$-th
  position and zeros everywhere else.
\end{proposition}

\begin{proof}
  We view both tangent spaces as subspaces of
  $T_P \PD_n = T_P \Sym^2(\RR^n) \cong \Sym^2(\RR^n)$.  As
  $\M(H)^{-1}$ is just the intersection of $\PD_n$ with the linear
  subspace of $\Sym^2(\RR^n)$ given by the vanishing of the
  $ij$-entries for all non-edges $ij$ of~$H$, the second claim
  follows.  For $T_P \M(G)$, we use the differential of the matrix
  inversion $\inv\colon \PD_n \rightarrow \PD_n$.  In coordinates
  $X = (x_{ij})$ it~is
  \begin{equation*}
    \frac{\partial \inv}{\partial x_{ij}}(X) =
    \frac{\partial(X^{-1})}{\partial x_{ij}} = - X^{-1} E^{ij} X^{-1} = - (X^{-1})^i \cdot (X^{-1})_j.
  \end{equation*}
  Now $T_{P^{-1}} \M(H)^{-1}$ is generated by $E^{ij}$ for $i=j$ or
  $ij \in E_H$, and the differential at $P^{-1}$ maps
  $T_{P^{-1}} \M(H)^{-1}$ isomorphically onto $T_P \M(G)$, and the
  image of $E^{ij}$ is~$-M^{ij}$.
\end{proof}

\begin{remark}\label{r:dim}
  The proof also shows $\dim (\M(G)) = |E_G| + n$. Moreover,
  $G \cup H = K_N$ if and only if
  $T_{\mathbbm{1}_n} \M(G) + T_{\mathbbm{1}_n} \M(H)^{-1} =
  T_{\mathbbm{1}_n} PD_n$, that is, if and only if the intersection is
  transverse at $P = \mathbbm{1}_n$. In particular, under this
  assumption $\M(G,H) = \M(G) \cap \M(H)^{-1}$ is smooth
  at~$\mathbbm{1}_n$ and thus at all positive diagonal matrices by
  means of the map $P \mapsto D P D$ for a fixed positive diagonal
  matrix~$D$.  This is a diffeomorphism of $\PD_n$, restricting to a
  homeomorphism $\M(G,H) \rightarrow \M(G,H)$, mapping smooth points
  to smooth points and $\mathbbm{1}_n$ to~$D^2$.
\end{remark}

\begin{remark}
  For $H = K_N$, the proof of \cref{p:tangentSpaces} implies that the
  matrix $M \coloneqq (M^{ij}_{kl})_{kl, ij \in \binom{N+1}{2}}$ is
  symmetric and invertible whenever $P$ is symmetric and invertible --
  a fact that might be difficult to prove directly. If $P$ is positive
  definite, $M$ is the inverse information matrix appearing, for
  example, in the Cram\'er--Rao inequality and
  in~\cite[Example~2.2]{drton2009likelihood}.
\end{remark}

The following theorem was suggested to us by Piotr Zwiernik.  It was
the starting point of our investigations and here we present a
differential-geometric proof.

\begin{theorem}\label{theorem:Piotr}
  If $G \cup H = K_N$ is the complete graph, then the model
  $\M(G,H)$ is smooth. In fact, the intersection of $\M(G)$ and
  $\M(H)^{-1}$ is transverse at every intersection point.
\end{theorem}

\begin{proof}
  The inverse information matrix $M$ at a positive definite matrix
  $P \in \PD_{n}$ is positive definite. The matrix
  $\left( M^{ij}_{kl} \right)_{kl,ij \in E_G \setminus E_H}$ is its
  principal $(E_G \setminus E_H \times E_G \setminus E_H)$-submatrix
  and therefore also positive definite and in particular
  invertible. This implies $T_P \M(G) + T_P \M(H)^{-1} = T_P \PD_n$,
  so the intersection of the smooth submanifolds $\M(G)$ and
  $\M(H)^{-1}$ of $\PD_n$ is transverse at every common
  point~$P$. Therefore, the intersection
  $\M(G,H) = \M(G) \cap \M(H)^{-1}$ is a smooth submanifold of $\PD_n$
  by~\cite[Chapter~1~\S5]{guillemin2010differential}.
\end{proof}

\begin{remark}\label{r:PiotrArgument}
  Multivariate centered Gaussian random vectors form a regular
  exponential family with mean parameter $\Sigma$ and natural
  parameter~$K=\Sigma^{-1}$ \cite[Chapter 3]{sundberg2019}.  According
  to \cite[Corollary 3.17]{sundberg2019}, a mixed parametrization
  $(\sigma_{ij},k_{st})_{ij\in A, st\in B}$ with $A\cup B = E_{K_{N}}$
  and $A\cap B = \emptyset$ is a valid parametrization for the
  exponential family.  In the situation of \cref{theorem:Piotr}, when
  $G \cup H = K_N$, the non-edges of $G$ and $H$ are disjoint.
  Therefore one could pick $A, B \subseteq E_{K_{N}}$ such that the
  non-edges of $G$ are contained in $B$ and the non-edges of $H$
  in~$A$.  However, when zero constraints on both mean and natural
  parameters are imposed, the result is in general not a regular
  exponential family.  Therefore smoothness results like
  \cref{theorem:Piotr} do not simply follow from general theory.
\end{remark}

The proof shows that in \cref{theorem:Piotr},
$\codim(\M(G,H)) = \codim(\M(G)) + \codim(\M(H))$ and so, regarding
dimensions, by \cref{r:dim}
\begin{align*}
  \dim(\M(G,H)) & = \frac{n^{2}+n}{2} - \left(\frac{n^{2}-n}{2} - |E_{G}| + \frac{n^{2}-n}{2} - |E_{H}| \right) \\
                & = |E_G| + |E_H| - \frac{n^{2}-3n}{2} \\
                & = |E_{G} \cap E_{H}| + n,
\end{align*}
where we have used
$|E_{G}| + |E_{H}| = \binom{n}{2} + |E_{G}\cap E_{H}|$.  The
intersection with $\PD_{n,1}$ satisfies
\begin{equation*}
  \dim(\M_1(G,H)) = \dim(\M(G,H) \cap \PD_{n,1})  = \dim(\M(G,H)) - n = |E_G \cap E_H|.
\end{equation*}

If one is not in the favorable situation $G \cup H = K_{N}$, the
dimension computation becomes more involved.  Let $G$ and $H$ now be
arbitrary graphs on~$N$, and denote by $E_G^\comp$ the edge complement
$\binom{N}{2} \setminus E_G$ and similarly for~$E_H$.  The following
lemma is a technical core for dimension bounds.

\begin{lemma}\label{l:pseudo_jacobian}
  With $\Sigma=(\sigma_{st})_{st \in \binom{N+1}{2}}$, let
  $f_{ij} = \sigma_{ij}$ for $ij \in E_H^\comp$ and
  $g_{kl} = \det(\Sigma_{N \setminus k, N \setminus l})$ for
  $kl \in E_G^\comp$.  Consider the differentials
  \[
    J_{H} = \left(\frac{\partial f_{ij}}{\partial
        \sigma_{st}}\right)_{ij \in E^{\comp}_{H} ,
      st\in\binom{N+1}{2}}
    \quad \text{ and } \quad
    J_{G} = \left(
      \frac{\partial g_{kl}}{\partial{\sigma_{st}}}
      \right)_{kl \in E^{\comp}_{G} ,
      st\in\binom{N+1}{2}}.
  \]
  For every $\Sigma \in \M(G,H)$, define
  $(|E_G^\comp| + |E_H^\comp|) \times \binom{N+1}{2}$ matrix
  $\widetilde{J}_\Sigma$ by stacking $J_{G}$ and $J_{H}$:
  \begin{equation*}
    \widetilde{J}_\Sigma \coloneqq
    \begin{pmatrix}
      J_{G}\\
      J_{H}
    \end{pmatrix}.
  \end{equation*}
  Then $\rk \widetilde{J}_\Sigma \geq \binom{n}{2} - |E_G \cap E_H|$.
\end{lemma}

\begin{proof}
  The kernel of $\widetilde{J}_\Sigma$ is the intersection of the
  kernels of $J_{G}$ and $J_{H}$.  The Zariski closures
  $\overline{\M(G)}$ and $\overline{\M(H)^{-1}}$ in
  $\Sym^2(\RR^n) \cap \GL(\RR^n)$ of $\M(G)$ and $\M(H)^{-1}$ are both
  irreducible smooth varieties, $\M(G)$ by \cref{p:MarkovianCI} and
  $\M(H)^{-1}$ because it is a linear space.
  The Zariski tangent spaces have been computed in
  \cref{p:tangentSpaces}.
  Using $\dim(U \cap V) = \dim(U) + \dim(V) - \dim(U + V)$ for
  finite-dimensional vector spaces $U$ and $V$ inside a common vector
  space, we  compute
  \begin{align*}
    \dim(\ker(\widetilde{J}_\Sigma)) &= \dim(\Span (E^{ij}: ij \in E_H \text{ or } i=j ) \cap \Span ( M^{kl}: kl \in E_G \text{ or } k=l )) \\ &= (|E_G|+n) + (|E_H|+n) \\ &- \dim(\Span (E^{ij}: ij \in E_H \text{ or } i=j) + \Span (M^{kl}: kl \in E_G \text{ or } k=l )) \\ &\leq |E_G \cap E_H| + n.
  \end{align*}
    In the last step we used that the dimension of the sum of the two vector
    spaces is at least $|E_H \cup (E_G \setminus E_H)| + n = |E_G \cup E_H|
    + n$ because the matrix $(M^{kl}_{st})_{kl, st \in E_G \setminus E_H}$
    is a principal submatrix of the inverse of the information matrix and
    therefore invertible.
\end{proof}

\begin{theorem}\label{dim_model}
	We have $\dim(\M(G,H)) \leq |E_G \cap E_H| + n$.
\end{theorem}

\begin{proof}
  The polynomials $f_{ij}$ and $g_{kl}$ from \cref{l:pseudo_jacobian}
  lie in the vanishing ideal
  $\mathcal{I}(\M(G,H)) \subseteq \RR[\sigma_{st}: st \in
  \binom{N+1}{2}]$, and hence in the prime ideal of every irreducible
  component $Z$ of the Zariski closure
  $\mathcal{V}(\mathcal{I}(\M(G,H)))$ inside the affine space
  $\Sym^2(\RR^n)$. Then the Jacobian matrix $J_\Sigma$ at
  $\Sigma \in Z$ of a generating set of $\mathcal{I}(Z)$ satisfies
  $\rk J_\Sigma \geq \rk \widetilde{J}_\Sigma$ by
  \cref{lemma:rkInequality}.  By \cref{l:pseudo_jacobian}, for
  $\Sigma \in \M(G,H)$,
  $\rk \widetilde{J}_\Sigma \geq \binom{n}{2} - |E_G \cap E_H|$.  Then
  Krull's principal ideal theorem implies that the Zariski tangent
  space satisfies $\dim(T_\Sigma Z) \geq \dim(Z)$, and so
  \begin{equation*}
    \dim(Z) \leq \dim(T_\Sigma Z) = \binom{n+1}{2} - \rk(J_\Sigma) \leq \binom{n+1}{2} - \rk(\widetilde{J}_\Sigma) \leq |E_G \cap E_H| + n. \qedhere
  \end{equation*}
\end{proof}

\begin{corollary}\label{corollary:dim_model}
  We have $\dim(\M_1(G,H)) \leq |E_G \cap E_H|$.
\end{corollary}

\begin{remark}
  The inequalities in \cref{corollary:dim_model} and \cref{dim_model}
  can be strict.  Example~\ref{ex:Nonsmooth} contains a model
  $\M_1(G,H)$ of dimension $1$ with $|E_G \cap E_H| = 2$. It is
  reducible with two irreducible components which intersect only in
  the identity matrix~$\mathbbm{1}_{4}$.
\end{remark}

Using the dimension bound we can show that $E_G \cap E_H = \emptyset$
if $\M_1(G,H)$ is zero-dimensional, i.e., a union of finitely many
points (including~$\mathbbm{1}_n$).  Indeed, if
$E_G \cap E_H \neq \emptyset$, then $\dim(\M_1(G,H)) \geq 1$ because
after some permutation of~$N$ one can assume that
$12 \in E_{G}\cap E_{H}$.  Then
\begin{equation*} \Sigma =
  \begin{pmatrix}
    1 & \sigma_{12} & 0 & \ldots & 0\\
    \sigma_{12} & 1 & 0 &\ldots & 0\\
    0 & 0 & 1 & \ldots & 0\\
    \vdots & \vdots & \vdots & \ddots & \vdots\\
    0 & 0 & 0 & \ldots & 1
  \end{pmatrix}
\end{equation*}
lies in $\M_{1}(G,H)$ for any $\sigma_{12}\in (-1,1)$.
\Cref{c:trivialModel} below strengthens this further.  In fact,
$E_G \cap E_H = \emptyset$ if and only if the model consists only of
$\mathbbm{1}_{n}$.

\begin{lemma}\label{lemma:bound_Zariski_tangent_space}
  Let $V \coloneqq \overline{\M(G,H)}$ and $\Sigma \in \M(G,H)$. Then
  $\dim(T_\Sigma V) \leq |E_G \cap E_H| + n$.
\end{lemma}

\begin{proof}
  The Zariski tangent space $T_\Sigma V$ is the kernel of
  the Jacobian matrix $J_\Sigma$ at $\Sigma \in \M(G,H)$ of a generating
  set of $\mathcal{I}(V)$.  By \cref{lemma:rkInequality} and
  \cref{l:pseudo_jacobian}, we have
  $\rk J_\Sigma \geq \rk \widetilde{J}_\Sigma \geq \binom{n}{2} - |E_G
  \cap E_H|$ which yields the claimed inequality.
\end{proof}

\begin{theorem}\label{theorem:dim_smooth}
  Every connected component of $\M(G,H)$ of dimension
  $|E_G \cap E_H| + n$ is smooth and has irreducible Zariski closure.
\end{theorem}

\begin{proof}
  Let $V \coloneqq \overline{\M(G,H)} \subseteq \Sym^2(\RR^n)$. Let $M$ be a connected component of $\M(G,H)$ and $Z$ an irreducible component of its Zariski closure $\overline{M}$ with $\dim(Z) = |E_G \cap E_H| + n$. Then at every point $\Sigma \in M \cap Z$, we have
  \begin{equation*}
    |E_G \cap E_H| + n = \dim(Z) \leq \dim(T_\Sigma Z) \leq \dim(T_\Sigma V) \leq |E_G \cap E_H| + n,
  \end{equation*}
  hence all inequalities are equalities, proving that the local rings
  $\mathcal{O}_{V,\Sigma}$ with $\Sigma \in M \cap Z$ are regular. Regular local rings are integral domains by \cite[Corollary~13.6]{kemper2011CommAlg}. Hence, for every $\Sigma \in M \cap Z$, there is only one
  irreducible component of $\overline{M}$ containing $\Sigma$, so $Z$ does not intersect any other irreducible component of $\overline{M}$ inside $M$.
  Therefore, as $M$ is connected, $\overline{M} = Z$ is irreducible, and $M = M \cap Z$ is smooth.
\end{proof}

\begin{remark}
  The proof of \Cref{theorem:dim_smooth} also shows that a connected component of $\M_1(G,H)$ of dimension $|E_G \cap E_H|$ is smooth and has irreducible Zariski closure. \Cref{proposition:edgeIntersection3} \eqref{item:smoothModelUnionNotComplete} contains a smooth model
  $\M_1(G,H)$ on $4$ vertices of dimension $3 = |E_G \cap E_H|$ with
  $G \cup H \neq K_4$, so the converse of \Cref{theorem:Piotr} is false. Even when $G \cup H \neq K_N$, \Cref{theorem:dim_smooth} still provides a sufficient criterion for smoothness. In fact, we know of no example of a smooth model $\M(G,H)$ (resp. $\M_1(G,H)$) having dimension less than $|E_G \cap E_H| + n$ (resp. $|E_G \cap E_H|$).
\end{remark}

We now move on to the vanishing ideal $\mathcal{I}(\M_{1}(G,H))$ of
double Markovian models.  Ordinary Gaussian graphical models have
rational parametrizations and their vanishing ideals are prime.
Vanishing ideals of double Markovian models need not be prime.  They
arise from conditional independence ideals by removing components
whose varieties do not intersect~$\PD_{n}$ and taking the radical.
We do not expect double Markovian CI ideals to be radical.  In the
discrete case, radicality fails even for ideals defined by the global
Markov condition~\cite[Example~4.9]{kahle2014positive}.  The~V\'amos
gaussoid from \cite[Example~13]{Geometry} yields a Gaussian CI ideal
which is not radical (in~the ring $\mathbb C[\sigma_{ij} : i \le j]$
where the diagonal of $\Sigma$ is \emph{not} normalized).

For the rest of the section we mostly restrict to the normalized
variance case, that is $\M_{1}(G,H)$ as opposed to~$\M(G,H)$.  This
removes duplication from the statements.  In most cases only small
changes are necessary to change a result for $\M_{1}(G,H)$ into one
for~$\M(G,H)$.

\begin{definition}\label{def_CI_ideal}
  Let $G$ and $H$ be two graphs on $N$ and $\Sigma = (\sigma_{ij})$ a
  generic symmetric matrix with ones on the diagonal.  The
  \emph{saturated conditional independence ideal}
  $\CI_{G,H} \subseteq \RR[\sigma_{ij}: i < j]$
  is the saturation of the ideal
  $\ideal{\sigma_{ij}, \det(\Sigma_{kC, lC}): ij \in E_H^\comp, kl \in
    E_G^\comp \text{ separated by } C}$ at the product of all
  principal minors of~$\Sigma$.  Similarly, the \emph{simplified
    saturated conditional independence ideal}
  $\SCI_{G,H} \subseteq \RR[\sigma_{ij}: i < j]$
  is the saturation of
  $\ideal{\sigma_{ij}, \det(\Sigma_{N \setminus l, N \setminus k}): ij
    \in E_H^\comp, kl \in E_G^\comp}$ at the product of all
  principal minors of~$\Sigma$.
\end{definition}

Clearly, $\SCI_{G,H} \subseteq \CI_{G,H}$ and their varieties in the
affine space of symmetric matrices with ones on the diagonal agree
over both $\RR$ and~$\CC$ by \Cref{lemma:Markovian}.
The ideals $\SCI_{G,H}$ and $\CI_{G,H}$ are saturations of determinantal
ideals of symmetric matrices.  The latter have been featured in the
work of Conca~\cite{conca1994divisor,conca1994symmetric,conca1994symmetricLadder},
but little seems to be known in general (at least in comparison to
ordinary determinantal ideals).  Double Markovian models might provide
an incentive to further study ideals generated by collections of
minors of sparse symmetric matrices.

In the following lemma we consider for a moment the scheme of
$\CI_{G,H}$, including the multiplicity information stored in the
coordinate ring.  The dual of the Zariski tangent space at the
identity $\mathbbm{1}_n$ is
$\mathfrak{m}/(\CI_{G,H} + \mathfrak{m}^2)$ where
$\mathfrak{m} \coloneqq \ideal{\sigma_{st}: s < t}$ is the maximal
ideal in $A \coloneqq \RR[\sigma_{st}: s < t]$ corresponding
to~$\mathbbm{1}_n$.  Its dimension is also known as the embedding
dimension of $(A/\CI_{G,H})_{\mathfrak{m}}$, denoted
$\edim((A/\CI_{G,H})_{\mathfrak{m}})$ and equals the dimension of the
tangent space at the identity.

\begin{lemma}\label{lemma:Zariski_tangent_space_at_1}
  Let $V \coloneqq \mathcal{V}(\CI_{G,H})$ as a subscheme of the
  affine space of symmetric $n \times n$ matrices with ones on the
  diagonal. Then the embedding dimension of $V$ at $\mathbbm{1}_n$ is
  $|E_G \cap E_H|$.
\end{lemma}

\begin{proof}
  The short exact sequence
  \begin{equation*}
    0 \rightarrow (\CI_{G,H} + \mathfrak{m}^2)/\mathfrak{m}^2 \rightarrow
    \mathfrak{m}/\mathfrak{m}^2 \rightarrow \mathfrak{m}/(\CI_{G,H} + \mathfrak{m}^2) \rightarrow 0
  \end{equation*}
  shows that
  \begin{equation*}
    \edim((A/\CI_{G,H})_{\mathfrak{m}}) = \binom{n}{2} -
      \dim_{\RR}((\CI_{G,H} + \mathfrak{m}^2)/\mathfrak{m}^2),
  \end{equation*}
  so it suffices to compute
  $d \coloneqq \dim_{\RR}((\CI_{G,H} +
  \mathfrak{m}^2)/\mathfrak{m}^2)$.  But $\mathfrak{m}^{2}$ contains
  all products of two or more $\sigma_{st}$ and therefore
  $\det(\Sigma_{kC, lC}) \equiv_{\mathfrak{m}^2} \pm \sigma_{kl}$.
  As every principal minor contains $1$ as a monomial, saturation does not change $d$. Hence, $d = |E_G^\comp \cup E_H^\comp| = \binom{n}{2} - |E_G \cap E_H|$.
\end{proof}

The lemma states $\dim(T_{\mathbbm{1}_{n}}V) = |E_G \cap E_H|$,
and similarly one shows $\dim(T_{\mathbbm{1}_n} V) = |E_G \cap E_H| + n$ if
$V = \mathcal{V}(\CI_{G,H})$ consists of all symmetric $n \times n$
matrices with no restriction on the diagonals, considering $\CI_{G,H}$
as an ideal in $\RR[\sigma_{st}: s \leq t]$.  The lemma also shows
that, when $\CI_{G,H}$ equals the vanishing ideal and $V$ is smooth
at $\mathbbm{1}_{n}$, then $\dim(\M_1(G,H)) = |E_G \cap E_H|$ and
$\dim(\M(G,H)) = |E_G \cap E_H| + n$.

The next proposition expresses almost-principal minors of symmetric
matrices via paths in a graph.  Here we use the conventions from
\cite{lauritzen1996graphical}.  A path $p$ traverses no vertex more
than once and $V(p) \subseteq N$ denotes this set of vertices,
$e(p) = kl$ the endpoints of $p$, and
$\sigma_{p} = \prod_{ij\in p}\sigma_{ij}$ the product over the
variables corresponding to edges of~$p$. The sign of $p$ is $\sgn(p) \coloneqq (-1)^{|V(p)| - 1}$.
With all this in place we can express almost-principal minors in terms
of path products.

\begin{proposition}
  \label{proposition:almostPrincipalMinorsViaPath}
  Let $H$ be a graph on the vertex set~$N$ and let
  $\Sigma = (\sigma_{ij})$ be a generic $n\times n$ symmetric matrix
  with $\sigma_{ij} = 0$ for all $ij \notin E_H$.
  Then
  \begin{equation*}
    (-1)^{k+l}\det(\Sigma_{N \setminus k, N \setminus l}) = \sum_{\substack{p \text{ path in } H, \\ e(p) = kl }} \sgn(p) \cdot \det(\Sigma_{N \setminus V(p), N \setminus V(p)}) \cdot \sigma_p.
  \end{equation*}
\end{proposition}

This formula appears first in \cite{jones2005covariance}.  We include
a quick proof disregarding the signs which we do not use in the
sequel.

\begin{proof}
  By the Leibniz formula,
  $\det(\Sigma_{N \setminus l, N \setminus k}) = \sum_{\tau}
  \sgn(\tau) \prod_{i \in N \setminus l} \sigma_{i, \tau(i)}$, where
  the sum is over all bijective
  $\tau\colon N\setminus l \to N\setminus k$.  The summand
  corresponding to $\tau$ is non-zero if and only if each
  $\set{i, \tau(i)}$ with $\tau(i) \neq i$ is an edge of~$H$. Starting
  at the vertex $k$, the sequence $k, \tau(k), \tau^2(k), \dots$ is a
  path from $k$ to $l$ in~$H$, showing that
  \begin{align*}
    \det(\Sigma_{N \setminus l, N \setminus k})
    &= \sum_{\substack{p \text{ path in } H, \\ e(p) = kl }} \pm \sigma_p \cdot \sum_{\tau': N \setminus V(p) \overset{\cong}{\rightarrow} N \setminus V(p)} \sgn(\tau') \prod_{i \in N \setminus V(p)} \sigma_{i, \tau'(i)} \\
    &= \sum_{\substack{p \text{ path in } H, \\ e(p) = kl }} \pm \sigma_p \cdot \det(\Sigma_{N \setminus V(p), N \setminus V(p)}).\qedhere
  \end{align*}
\end{proof}

If all but one term in $\det(\Sigma_{N\setminus k, N\setminus l})$
vanish, the CI~ideal ought to be a monomial ideal.  Taking care of
details like saturation, it also follows that it agrees with the
vanishing ideal.

\begin{theorem}\label{theorem:CI_ideal}
  Let $G$ and $H$ be graphs on $N$ such that for every non-edge $kl$
  of $G$, there is at most one path $p$ in $H$ connecting $k$
  and~$l$. Then
  \[
    \SCI_{G,H} = \CI_{G,H} = \mathcal{I}(\M_1(G,H))
    =
    \ideal{\sigma_{ij}, \sigma_p: ij \not\in E_H,\; p \text{ a path in } H \text{ such that } e(p) \not\in E_G}.
  \]
\end{theorem}

\begin{proof}
  Let $\Sigma$ be the generic symmetric matrix with ones on the
  diagonal and zeros corresponding to non-edges of~$H$.
  If $ij \notin E_H$, all terms of almost-principal minors of $\Sigma$ which contain
  $\sigma_{ij}$ can be neglected since
  $\sigma_{ij} \in \SCI_{G,H} \subseteq \CI_{G,H}$.  By
  Proposition~\ref{proposition:almostPrincipalMinorsViaPath}, if $p$
  is the unique path connecting $k$ and $l$ inside $H$, then
  $\det(\Sigma_{N \setminus k, N \setminus l}) = \pm \det(\Sigma_{N
    \setminus V(p), N \setminus V(p)}) \cdot \sigma_p$, so
  $\sigma_{p}$ lies in the saturated simplified conditional independence
  ideal.  If there exists no such path, this almost-principal minor
  vanishes. Since the square-free monomial ideal
  \begin{equation*}
    \ideal{\sigma_{ij}, \sigma_p: ij \not\in E_H, p \text{ a path in } H \text{ such that } e(p) \not\in E_G}
  \end{equation*}
  agrees with its saturation at the product of all principal minors of
  $\Sigma$, it equals~$\SCI_{G,H}$.

  To show that this (radical) ideal equals the vanishing ideal, it
  suffices to see that each of its components intersects~$\PD_{n,1}$
  in a smooth real point.  This is clear since all irreducible
  components of $\mathcal{V}_{\CC}(\SCI_{G,H})$ are coordinate
  subspaces.
\end{proof}

The hypothesis of Theorem~\ref{theorem:CI_ideal} may seem restrictive
but, for example, it includes the case that $H$ is a forest and $G$ is
arbitrary.  On the other hand, it is easy to find an example on four
vertices where $H$ is a cycle and $\SCI_{G,H}$ is not a monomial ideal.

Determining vanishing ideals of ordinary Gaussian graphical models can
already be complicated, see for example~\cite{misra2020gaussian}.
However, it seems plausible that a divide-and-conquer approach based
on toric fiber products as in \cite{TFP-II,sullivant07:_toric} is
applicable for some suitably decomposable graph pairs~$G,H$.

\subsection{Connectedness}\label{sec:connectedness}
In this subsection we study connectedness of $\M(G,H)$ in the real
topology.  If the vanishing ideal is known and simple enough, the
results are easy as in the next corollary.  \Cref{thm:connected}
contains a sufficient condition based on connectedness in $G$ and~$H$.

\begin{corollary}\label{corollary:coordinate_planes}
  Under the hypotheses of Theorem~\ref{theorem:CI_ideal}, the model
  $\M(G,H)$ is connected.   Moreover the following are equivalent.
  \begin{enumerate}
  \item $\M(G,H)$ is smooth.
  \item $\M(G,H)$ is irreducible.
  \item $\M(G,H)$ has the maximal dimension $|E_G \cap E_H| + n$.
  \end{enumerate}
  In this case, $\M(G,H) = \M(G \cap H)^{-1}$ is an inverse graphical model, hence a spectrahedron.
\end{corollary}

\begin{proof}
  As a union of coordinate subspaces intersected with $\PD_n$,
  $\M(G,H)$ is star-shaped with respect to the identity matrix and
  thus connected.  Therefore, having the maximal dimension implies
  smoothness by Theorem~\ref{theorem:dim_smooth}, and smoothness together with connectedness implies irreducibility because regular local rings are integral domains~\cite[Corollary~13.6]{kemper2011CommAlg}.
  By Theorem~\ref{theorem:CI_ideal}, irreducibility implies that the
  square-free monomial ideal $\SCI_{G,H}$ is prime.  This is
  equivalent to the condition that for a path $p$ inside $H$ with
  $e(p) = kl \in E_G^\comp$ there exists an edge
  $ij \in p \cap E_G^\comp$. Thus, if $G'$ is the graph on $N$ which
  is obtained from $G$ by adding all edges in
  $E_G^\comp \cap E_H^\comp$, then
  $\M(G,H) = \M(G',H) = \M(G \cap H)^{-1}$, in particular $\M(G,H)$
  has dimension $|E_G \cap E_H| + n$.
\end{proof}

\begin{theorem}\label{thm:connected}
  Let $G$ and $H$ be graphs on $N$ with the property that there exists
  $i \in N$ such that for all non-edges $kl \in E_G^\comp$, every path
  in $H$ connecting $k$ and $l$ contains~$i$. Then the model $\M(G,H)$
  is connected.
\end{theorem}

\begin{proof}
  In this proof we denote by $H \setminus i$ the graph on $N$ obtained
  from $H$ by deleting all edges incident with~$i$ but keeping $i$ as a vertex. The model
  $\M(H \setminus i)^{-1}$ is connected as it is the intersection of a
  linear space with the convex set~$\PD_n$, and the intersection of
  convex sets is convex, hence connected.  Clearly,
  $\Sigma \in \M(H \setminus i)^{-1}$ if and only if
  $\Sigma \in \M(H)^{-1}$ and $\sigma_{ij} = 0$ for all $j \neq
  i$. The determinantal identity of
  Proposition~\ref{proposition:almostPrincipalMinorsViaPath}, and the
  assumptions on $G$ and $H$ imply
  $\M(H \setminus i)^{-1} \subseteq \M(G,H)$. Now, let
  $\Sigma \in \M(G,H)$ be arbitrary.  It suffices to find a path from
  $\Sigma$ to some matrix in~$\M(H \setminus i)^{-1}$.

  For $\varepsilon \in [0,1]$ consider
  \begin{equation*}
    \Sigma^\varepsilon \coloneqq \Sigma \odot
    \begin{pmatrix}
      1 & \dots & 1 & \varepsilon & 1 & \dots & 1 \\
      \vdots & & \vdots & \vdots &  \vdots & & \vdots\\
      1 & \dots & 1 & \varepsilon & 1 & \dots & 1 \\
      \varepsilon & \dots & \varepsilon & 1 & \varepsilon & \dots & \varepsilon \\
      1 & \dots & 1 & \varepsilon & 1 & \dots & 1 \\
      \vdots &  &  \vdots & \vdots &  \vdots & &  \vdots \\
      1 & \dots & 1 & \varepsilon & 1 & \dots & 1
    \end{pmatrix},
  \end{equation*}
  where the second factor has entries $\varepsilon$ in the $i$-th row
  and column, but $1$ in entry~$ii$.  The symbol $\odot$ denotes the
  Hadamard product which multiplies matrices entry-wise.  Then
  $\Sigma^\varepsilon$ is symmetric and positive definite for all
  $\varepsilon \in [0,1]$ as it is the Hadamard product of a positive
  definite matrix and a positive semi-definite matrix with strictly
  positive diagonal entries.  Moreover, $\Sigma^1 = \Sigma$ and
  $\Sigma^0 \in \M(H \setminus i)^{-1}$, so it suffices to show that
  $\Sigma^\varepsilon \in \M(G,H)$ for all $\varepsilon \in
  [0,1]$. This~follows from the assumptions on $G$ and $H$ as for all
  $kl \in E_G^\comp$ with $k \neq i \neq l$ we have
  \begin{align*}
    \det((\Sigma^\varepsilon)_{N \setminus k, N \setminus l}) &= \sum_{\substack{p \text{ path in } G \\ e(p) = kl}} \sgn(p) \cdot \sigma^\varepsilon_p \cdot \det((\Sigma^\varepsilon)_{N \setminus V(p), N \setminus V(p)}) \\
                                                              &= \varepsilon^2 \cdot \sum_{\substack{p \text{ path in } G \\ e(p) = kl}} \sgn(p) \cdot \sigma_p \cdot \det(\Sigma_{N \setminus V(p), N \setminus V(p)}) \\
                                                              &= \varepsilon^2 \cdot \det(\Sigma_{N \setminus k, N \setminus l}) = 0,
  \end{align*}
  using in the second step that, by assumption, $i \in V(p)$ for all
  occurring paths $p$, so that the occurring principal minors of
  $\Sigma^\varepsilon$ agree with the corresponding principal minors
  of~$\Sigma$.  Moreover, each monomial $\sigma^\varepsilon_p$
  contains exactly two variables that are scaled by~$\varepsilon$. If
  one of $k$ and $l$ agrees with~$i$, the same calculation works with
  $\varepsilon$ instead of~$\varepsilon^2$.
\end{proof}

\subsection{The decomposition theorem}
\label{sec:structure-theorem}

Section~\ref{sec:classifyByCommonE} contains a classification of
models with small $|E_{G}\cap E_{H}|$.  This is based on the following
decomposition theorem, whose proof also works for complex hermitian
positive definite matrices.

\begin{theorem}\label{theorem:Aida}
  Let $G$, $H$ be two graphs on the vertex set~$N$. Let
  $V_1, \ldots, V_r$ be a partition of~$N$ such that each $V_i$ is the
  vertex set of a connected component of $G \cap H$, considered as the
  graph on $N$ with edge set $E_G \cap E_H$. Then
  \begin{equation*}
    \M(G,H) = \bigoplus_{i=1}^r \M(G|_{V_i}, H|_{V_i}).
  \end{equation*}
  In words, every $\Sigma \in \M(G,H)$ is block-diagonal with $r$
  blocks whose rows and columns are indexed by the~$V_i$.
\end{theorem}

\begin{proof}
  Inductively, it suffices to consider the case $r = 2$. We set
  $V \coloneqq V_1$ and $W \coloneqq V_2$. It is then enough to show
  that every matrix
	\begin{equation*}
          \Sigma =
          \left(\begin{array}{c|c}
                  \Sigma_{VV} & \Sigma_{VW} \\ \hline
                  \Sigma_{VW}^t & \Sigma_{WW}
	\end{array}\right) \in \M(G,H)
	\end{equation*}
	is block-diagonal, i.e., $\Sigma_{VW} = 0$.  We partition
        $\Sigma^{-1}$ in the same way as~$\Sigma$. By assumption,
        $G \cap H$ contains no edges between $V$ and $W$. This implies
        that the matrix $\Sigma_{VW} (\Sigma^{-1})_{VW}^t$ has only
        zeros on the diagonal as for all $v \in V$,
        \[
          (\Sigma_{VW} (\Sigma^{-1})_{VW}^t)_{vv} = \sum_{w \in W}
          \Sigma_{vw} (\Sigma^{-1})_{vw} = \sum_{w \in W} 0 = 0.
        \]
        In particular, $\tr(\Sigma_{VW} (\Sigma^{-1})_{VW}^t) = 0$. As
        $\Sigma_{VV}$ and $\Sigma_{WW}$ are positive definite, there
        exist symmetric square roots $A_V$ and $A_W$ such that
        $A_V^2 = \Sigma_{VV}$ and $A_W^2 = \Sigma_{WW}$. We now define
	\begin{equation*}
	\Sigma' \coloneqq
	\left(\begin{array}{c|c}
	A_V^{-1} & 0 \\ \hline
	0 & A_W^{-1}
	\end{array}\right) \cdot
	\Sigma \cdot
	\left(\begin{array}{c|c}
	A_V^{-1} & 0 \\ \hline
	0 & A_W^{-1}
	\end{array}\right) =
	\left(\begin{array}{c|c}
	\mathbbm{1}_V & A_V^{-1} \Sigma_{VW} A_W^{-1} \\ \hline
	A_W^{-1} \Sigma_{VW}^t A_V^{-1} & \mathbbm{1}_W
	\end{array}\right).
	\end{equation*}
	Clearly, $\tr(\Sigma') = |V| + |W| = n$. For the inverse matrix we have
	\begin{equation*}
	\Sigma'^{-1} =
	\left(\begin{array}{c|c}
	A_V & 0 \\ \hline
	0 & A_W
	\end{array}\right) \cdot
	\Sigma^{-1} \cdot
	\left(\begin{array}{c|c}
	A_V & 0 \\ \hline
	0 & A_W
	\end{array}\right) =
	\left(\begin{array}{c|c}
	(\Sigma'^{-1})_{VV} & A_V (\Sigma^{-1})_{VW} A_W \\ \hline
	A_W (\Sigma^{-1})_{VW}^t A_V & (\Sigma'^{-1})_{WW}
	\end{array}\right).
	\end{equation*}
	Now observe that
        $\Sigma'_{VW} (\Sigma'^{-1})_{VW}^t = A_V^{-1} \Sigma_{VW}
        (\Sigma^{-1})_{VW}^t A_V$ as the product
        $A_W^{-1} A_W = \mathbbm{1}_W$ in the middle cancels out.
        Since the trace is cyclic, we have
	\begin{equation*}
	\tr(\Sigma'_{VW} (\Sigma'^{-1})_{VW}^t) = \tr(A_V^{-1} \Sigma_{VW} (\Sigma^{-1})_{VW}^t A_V) = \tr(\Sigma_{VW} (\Sigma^{-1})_{VW}^t A_V A_V^{-1}) = 0.
	\end{equation*}
	Moreover,
	\begin{equation*}
	\mathbbm{1}_V = (\Sigma' \Sigma'^{-1})_{VV} = (\Sigma'^{-1})_{VV} + \Sigma'_{VW} (\Sigma'^{-1})_{VW}^t,
	\end{equation*}
	implying $\tr((\Sigma'^{-1})_{VV}) = \tr(\mathbbm{1}_V) =
        |V|$. Similarly, $\tr((\Sigma'^{-1})_{WW}) = |W|$, so
        $\tr(\Sigma'^{-1}) = |V| + |W| = n$. As $\Sigma'$ is real
        symmetric positive definite there exists a symmetric
        square root~$T$ with $T^2 = \Sigma'$ and
        thus also $T^{-2} = \Sigma'^{-1}$.  Using the inner
        product
        $\langle X, Y \rangle = \tr(XY) = \sum_{i,j = 1}^n x_{ij}
        y_{ij}$ on the space of real symmetric matrices,
	\begin{gather*}
	\langle T, T \rangle = \tr(T^2) = \tr(\Sigma') = n, \\
	\langle T^{-1}, T^{-1} \rangle = \tr(T^{-2}) = \tr(\Sigma'^{-1}) = n, \\
	\langle T, T^{-1} \rangle = \tr(\mathbbm{1}_n) = n.
	\end{gather*}
	Therefore, we have equality in the Cauchy--Schwarz inequality
	\begin{equation*}
	n^2 = \langle T, T^{-1} \rangle^2 \leq \langle T, T \rangle \cdot \langle T^{-1}, T^{-1} \rangle = n^2,
	\end{equation*}
	implying that $T$ and $T^{-1}$ are linearly dependent as
        matrices, i.e., $T = \lambda T^{-1}$ for some
        $\lambda \in \RR$. This implies
        $\Sigma' = T^2 = \lambda \mathbbm{1}_n$. In particular,
        $0 = \Sigma'_{VW} = A_V^{-1} \Sigma_{VW} A_W^{-1}$ as
        matrices. But this is equivalent to $\Sigma_{VW} = 0$, as
        desired.
\end{proof}

If all $V_i$ are single vertices we get the following.

\begin{corollary}\label{c:trivialModel}
	We have $\M_1(G,H) = \set{\mathbbm{1}_n}$ if and only if
	$E_G \cap E_H = \emptyset$.
\end{corollary}

In other words, if $\Sigma$ is a symmetric positive definite
$(n \times n)$-matrix with the property that every off-diagonal
entry vanishes either in $\Sigma$ or in $\Sigma^{-1}$ (or
both), then $\Sigma$ is a diagonal matrix.  We have not found this
result in the literature.

\begin{remark}
  A natural question is whether the assumption of positive
  definiteness in \cref{theorem:Aida} is necessary.
  \Cref{ex:ZeroDimlModel} shows that \cref{c:trivialModel} does not
  hold for positive \emph{semi}-definite matrices, and
  \cref{ex:CounterexamplePrinReg} below shows that \cref{theorem:Aida}
  does not hold for principally regular matrices, that is, matrices
  whose principal minors do not vanish.  We do not know if
  \cref{c:trivialModel} holds in this case.
\end{remark}

\begin{remark}
  A simpler variant of \cref{theorem:Aida} can be proven by recursive
  direct sum decomposition and duality: To every pair of graphs
  $(G,H)$ on $N$ there exists a partition $V_1,\dotsc,V_r$ of $N$ such
  that
  \begin{enumerate}
  \item $G_i = G|_{V_i}$ and $H_i = H|_{V_i}$ are connected.
  \item $\M(G,H)$ is smooth if and only if all $\M(G_i,H_i)$ are smooth.
  \item $\M(G,H)$ is connected if and only if all $\M(G_i,H_i)$ are connected.
  \end{enumerate}
  The merit of this simpler assertion is that it does not require
  positive definiteness.  It also holds for principally regular models
  of~$\CIS{G,H}$ over $\CC$ because the proof uses only elementary
  operations on CI~relations introduced in \Cref{sec:MinorsEtc}.
\end{remark}
\begin{figure}[tpb]
  \centering
		\begin{tikzpicture}[scale=1.4]
		\tikzset{every node/.style={draw, circle, inner sep=2pt, thick}}
		\tikzset{every edge/.style={draw, ultra thick}}
		\tikzset{G/.style={green!70!black}}
		\tikzset{H/.style={blue!70!black}}
		
		\node (n1) at (1, 1) {$1$};
		\node (n2) at (1, 0) {$2$};
		\node (n3) at (0, 0) {$3$};
		\node (n4) at (3, 1) {$4$};
		\node (n5) at (3, 0) {$5$};
		\node (n6) at (4, 0) {$6$};
		
		\path (n1) edge (n2);
		\path (n2) edge (n3);
		\path (n3) edge (n1);
		
		\path (n4) edge (n5);
		\path (n5) edge (n6);
		\path (n6) edge (n4);
		
		\path (n1) edge [H] (n4);
		\path (n1) edge [H] (n5);
		\path (n3) edge [H, bend left=70] (n4);
		\path (n3) edge [H, bend right=40] (n6);
		
		\path (n1) edge [G, bend left=70] (n6);
		\path (n2) edge [G] (n4);
		\path (n2) edge [G] (n5);
		\path (n2) edge [G, bend right] (n6);
		\path (n3) edge [G, bend right] (n5);
		\end{tikzpicture}
                \caption{The graphs for
                  \cref{ex:CounterexamplePrinReg}.  The edges of
                  $G \cap H$ are drawn black, $G \setminus H$ in green
                  and $H \setminus G$ in blue. Then
                  $G \cap H = K_{123} \oplus K_{456}$ and
                  $G \cup H = K_{1\cdots 6}$.}
  \label{fig:greenblue}
\end{figure}
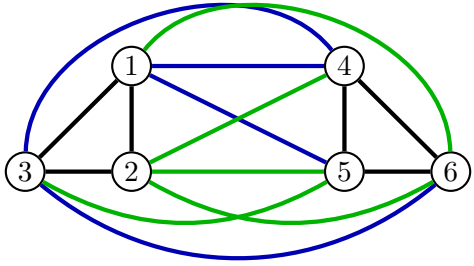

\begin{example}\label{ex:CounterexamplePrinReg}
  Consider the graph in \cref{fig:greenblue}.  We study the variety of
  $\CIS{G,H}$ in \texttt{Macaulay2}:
  \enlargethispage{\baselineskip}
  \begin{minted}{macaulay2}
	R = QQ[x11,x12,x13,x14,x15,x16,  x22,x23,x24,x25,x26,
	               x33,x34,x35,x36,          x44,x45,x46,
	                       x55,x56,                  x66]
	X = genericSymmetricMatrix(R,x11,6)

	-- Impose the relations from $H$ directly on the matrix
	X = sub(X, { x16=>0, x24=>0, x25=>0, x26=>0, x35=>0 })

	-- Pick an affine slice of the model which is likely to contain
	-- positive definite matrices by diagonal dominance
	X = sub(X, {
	  x11=>10, x22=>10, x33=>10, x44=>10, x55=>10, x66=>10,
	  x12=>1,  x13=>1,  x23=>1,  x45=>1,  x46=>1,  x56=>1
	})
	\end{minted}

	\[
	\begin{pmatrix}
	10     & 1  & 1      & x_{14} & x_{15} & 0      \\
	1      & 10 & 1      & 0      & 0      & 0      \\
	1      & 1  & 10     & x_{34} & 0      & x_{36} \\
	x_{14} & 0  & x_{34} & 10     & 1      & 1      \\
	x_{15} & 0  & 0      & 1      & 10     & 1      \\
	0      & 0  & x_{36} & 1      & 1      & 10
	\end{pmatrix}
	\]
	Some of the variables are specified to ensure quick
        termination of the following computations.
        If \cref{theorem:Aida} held for principally regular matrices,
        $x_{14}$, $x_{15}$, $x_{34}$ and $x_{36}$ would vanish on
        every principally regular matrix satisfying the equations
        of~$\CIS{G}$.

    \vskip 0.5em
	\begin{minted}{macaulay2}
	-- The relations imposed by $G$
	I = radical ideal(
	  det submatrix'(X, {0}, {3}),  -- $14$
	  det submatrix'(X, {0}, {4}),  -- $15$
	  det submatrix'(X, {2}, {3}),  -- $34$
	  det submatrix'(X, {2}, {5})   -- $36$
	)
	-- Saturation at each of the principal minors
	J = fold((I,f) -> I : f, I, subsets(numRows(X)) / (K -> det X_K^K))
	decompose J
	\end{minted}
	\begin{align*}
	& \ideal{x_{14},x_{15},x_{34},x_{36}} \\
	\cap\; & \ideal{1210 x_{14}^2-999,-11 x_{14}+x_{15},-x_{14}+x_{34},-11 x_{14}+x_{36}} \\
	\cap\; & \ideal{1210 x_{14}^2-981,-11 x_{14}+x_{15},x_{14}+x_{34},11 x_{14}+x_{36}}
	\end{align*}
	
	The first component has the desired block structure of
        $K_{123} \oplus K_{456}$, but the other components contain
        real points as well. Consider the last component. It consists
        of two real points:
	\begin{equation*}
          x_{14} = \pm \sqrt\frac{981}{1210},\quad
	x_{15} = 11 x_{14},\quad
	x_{34} = - x_{14},\quad
	x_{36} = -11 x_{14}.
	\end{equation*}
	This yields a real matrix satisfying the equations of
        $\CIS{G,H}$ and whose principal minors are non-zero. However,
        the determinant of the entire matrix equals
        $-\frac{4374}{55}$, which is not positive.
	
	This shows that $\CIS{G,H}$ has real, principally regular
        solutions without block-diagonal structure. The positive
        definite matrices in the affine slice $J$ of the model all
        fall into the first component and do have the block
        structure. A purely algebraic computation, without taking
        positive definiteness into account, would not be able to prove
        \cref{theorem:Aida}.
\end{example}
\subsection{Classification of the models
  \texorpdfstring{$\M(G,H)$}{M(G,H)} with
  \texorpdfstring{$|E_G \cap E_H| \leq 3$}{EG cap EH}}
\label{sec:classifyByCommonE}
\Cref{dim_model} bounds the model dimension in terms of
$|E_{G}\cap E_{H}|$.  We finish our analysis of the geometry of
$\M(G,H)$ with a classification of models with small intersections of the edge sets. In view of \cref{theorem:Aida} we can restrict to the cases where $E_G \cap E_H$ defines a connected graph on the subset of vertices of $N$ incident to some edge in $E_G \cap E_H$. For any $N' \subseteq N$, we also write $\PD_{N'}$ for the set of positive definite matrices with rows and columns indexed by $N'$. For disjoint subsets $N'$ and $N''$, a direct sum $\PD_{N'} \oplus \PD_{N''}$ indicates the set of block-diagonal positive definite matrices inside $\PD_{N' \cup N''}$ with the rows and columns of the two blocks indexed, respectively, by $N'$ and~$N''$, and similarly for $\PD_{N',1} \oplus \PD_{N'',1}$ if we restrict to ones on the diagonal.

\begin{proposition}\label{proposition:single_edge}
  Let $E_G \cap E_H = \set{ ij }$ consist of a single edge. Then
  \begin{equation*}
    \M(G,H) = \PD_{ij}.
  \end{equation*}
  In particular, $\M_1(G,H)$ is connected and smooth of the maximal dimension~$|E_G \cap E_H| = 1$.
\end{proposition}

\begin{proof}
	Immediate from \cref{theorem:Aida} and the definitions.
\end{proof}

\begin{proposition}\label{proposition:classify}
  Let $|E_G \cap E_H| = 2$ and so $E_G \cap E_H = \set{ ij, jk }$ with
  distinct $i,j,k$.
  \begin{enumerate}
  \item\label{it:eqone} If $ik \in E_G \setminus E_H$, then
    $\M(G,H) = \M(G\cap H)^{-1}$ is an inverse graphical model.
  \item\label{it:eqtwo} In case $ik \in E_H \setminus E_G$, symmetrically $\M(G,H) = \M(G \cap H)$ is a graphical model.
  \item\label{it:eqthree} If $ik \not\in E_G \cup E_H$, then $\M_1(G,H)$ decomposes as
    \begin{equation*}
      \M_1(G,H) = \{ \mathbbm{1}_n + t E^{ij}: t \in (-1, 1) \} \cup \{ \mathbbm{1}_n + t E^{jk}: t \in (-1, 1) \}.
    \end{equation*}
    The Zariski closure of $\M_1(G,H)$ is a pair of lines intersecting
    in $\mathbbm{1}_n$.  Thus $\M_1(G,H)$ is connected of dimension
    one, with reducible Zariski closure.
  \end{enumerate}
\end{proposition}

\begin{proof}
  After a suitable permutation of $N$, we
  assume $ij = 12$ and $jk = 23$.  Any $\Sigma \in \M_1(G,H)$ has the
  block-diagonal form consisting of an upper left $(3 \times 3)$-block
  and an identity matrix. Therefore we can assume $N = 123$.
  Then, in the first case, we have
  \begin{equation*}
    \Sigma =
    \begin{pmatrix}
      1 & \sigma_{12} & 0 \\
      \sigma_{12} & 1 & \sigma_{23} \\
      0 & \sigma_{23} & 1
    \end{pmatrix},
    \hspace{5pt}
    \adj(\Sigma) =
    \begin{pmatrix}
      1 - \sigma_{23}^2 & -\sigma_{12} & \sigma_{12} \sigma_{23} \\
      -\sigma_{12} & 1 & -\sigma_{23} \\
      \sigma_{12} \sigma_{23} & -\sigma_{23} & 1 - \sigma_{12}^2
    \end{pmatrix}.
  \end{equation*}
  As $13 \in E_G$, we have that $G = K_3$ is complete, so there are no further
  restrictions and we obtain
  $\M(G,H) = \M(G \cap H)^{-1}$. In the second case, the same is true if we replace $\Sigma$ by $\Sigma^{-1}$ everywhere, so $\M(G,H) = \M(G \cap H)$. Finally, in the third case, $13 \not\in E_G \cup E_H$, so we additionally get $\sigma_{12} = 0$ or
  $\sigma_{23} = 0$, obtaining the union of two line segments, as desired.
\end{proof}

\begin{remark}
  Proposition~\ref{proposition:classify} shows that in some
  non-obvious cases double Markovian models are graphical or inverse
  graphical models.  This theme has occurred in the literature.  For
  example, in \cite[Proposition~12]{drton08a} it is shown that the
  only way that a covariance graph model $\M(K_{N},H)$ is a graphical
  model $\M(G,K_{N})$ is if covariance and concentration matrices have
  aligned block structures and the model is a product of $\PD$ cones
  (in particular, $G=H$ is a disjoint union of cliques).
\end{remark}

The same ideas also prove the following via direct computations.\footnote{\Cref{proposition:edgeIntersection3} differs slightly from the published version \href{https://doi.org/10.1111/sjos.12604}{\ttfamily doi:10.1111/sjos.12604} because the latter did not include the case where $E_G \cap E_H$ forms a star.}

\begin{proposition}\label{proposition:edgeIntersection3}
	Let $|E_G \cap E_H| = 3$. If $E_G \cap E_H = \set{ij,ik,jk}$ forms a
	$3$-clique, then $\M(G,H) = \PD_{ijk}$. Otherwise
	$E_G \cap E_H = \set{ij,jk,kl}$ with distinct $i,j,k,l$ forms a
	path or $E_G \cap E_H = \set{ij,ik,il}$ is a star.
	Up to swapping $G$ and $H$, we can restrict to the case where
	$H|_{ijkl}$ has equally many or more non-edges than~$G|_{ijkl}$ (i.e. at least as many prescribed zeros in the covariance matrix as in the concentration matrix).
	We can restrict moreover to $n = 4$ and assume $(i,j,k,l) = (1,2,3,4)$. If $E_G \cap E_H = \{12,23,34\}$ is a path, we have the following cases up to symmetry and inversion:
	\begin{enumerate}
		\item $E_H = E_G \cap E_H$ and $E_G = K_{1234}$. Here, $\M(G,H) = \M(G \cap H)^{-1}$.
		\item  $E_H = E_G \cap E_H$ and $E_G = K_{1234} \setminus \set{13}$. Here, $\M(G,H) = (\PD_{12} \oplus \PD_{34}) \cup \M(\set{23,34})^{-1}$.
		\item $E_H = E_G \cap E_H$ and $E_G = K_{1234} \setminus \set{14}$. Here, $\M(G,H) = (\PD_{12} \oplus \PD_{34}) \cup \M(\set{12,23})^{-1} \cup \M(\set{23,34})^{-1}$.
		\item $E_H = E_G \cap E_H$ and $E_G = K_{1234} \setminus \set{13,14}$. Here, $\M(G,H)$ is as in the previous case.
		\item $E_H = E_G \cap E_H$ and $E_G = K_{1234} \setminus \set{13,24}$. Here, $\M(G,H) = (\PD_{12} \oplus \PD_{34}) \cup \PD_{23}$.
		\item $E_G = E_H = E_G \cap E_H$. Here, $\M(G,H)$ is as in the previous case.
		\item $E_H = (E_G \cap E_H) \cup \set{13}$ and $E_G = K_{1234} \setminus \set{13}$. Here,
		\begin{equation*}
			\M_1(G,H) = \set{
				\begin{pmatrix}
					1 & \sigma_{12} & \sigma_{12}\sigma_{23} & 0 \\
					\sigma_{12} & 1 & \sigma_{23} & 0 \\
					\sigma_{12}\sigma_{23} & \sigma_{23} & 1 & \sigma_{34} \\
					0 & 0 & \sigma_{34} & 1
				\end{pmatrix}\colon
				\sigma_{12} \in (-1,1), \sigma_{23}^2 + \sigma_{34}^2 < 1
			}.
		\end{equation*}
		\item\label{item:smoothModelUnionNotComplete} $E_H = (E_G \cap E_H) \cup \set{13}$ and $E_G = K_{1234} \setminus \set{13,14}$. Here, $\M_1(G,H)$ is as in the previous case.
		\item $E_H = (E_G \cap E_H) \cup \set{13}$ and $E_G = K_{1234} \setminus \set{13,24}$. Here,
		\begin{align*}
			\M_1(G,H) &= (\PD_{\set{12},1} \oplus \PD_{\set{34},1}) \\
			&\cup \set{
				\begin{pmatrix}
					1 & \sigma_{12} & \sigma_{12}\sigma_{23} & 0 \\
					\sigma_{12} & 1 & \sigma_{23} & 0 \\
					\sigma_{12}\sigma_{23} & \sigma_{23} & 1 & 0 \\
					0 & 0 & 0 & 1
				\end{pmatrix}\colon
				\sigma_{12}, \sigma_{23} \in (-1,1)
			}.
		\end{align*}
		\item $E_H = (E_G \cap E_H) \cup \set{14}$ and $E_G = K_{1234} \setminus \set{14}$. Here,
		\begin{equation*}
			\M_1(G,H) = \set{
				\begin{pmatrix}
					1 & \sigma_{12} & 0 & \frac{-\sigma_{12} \sigma_{23} \sigma_{34}}{1-\sigma_{23}^2} \\
					\sigma_{12} & 1 & \sigma_{23} & 0 \\
					0 & \sigma_{23} & 1 & \sigma_{34} \\
					\frac{-\sigma_{12} \sigma_{23} \sigma_{34}}{1-\sigma_{23}^2} & 0 & \sigma_{34} & 1
				\end{pmatrix}\colon
				\sigma_{12}^2 + \sigma_{23}^2 < 1, \sigma_{23}^2 + \sigma_{34}^2 < 1
			}.
		\end{equation*}
		\item $E_H = (E_G \cap E_H) \cup \set{14}$ and $E_G = K_{1234} \setminus \set{13,14}$. Here,
		\begin{align*}
			\M_1(G,H) &= (\PD_{\set{12},1} \oplus \PD_{\set{34},1}) \cup \set{
				\begin{pmatrix}
					1 & 0 & 0 & 0 \\
					0 & 1 & \sigma_{23} & 0 \\
					0 & \sigma_{23} & 1 & \sigma_{34} \\
					0 & 0 & \sigma_{34} & 1
				\end{pmatrix}\colon
				\sigma_{23}^2 + \sigma_{34}^2 < 1
			}.
		\end{align*}
	\end{enumerate}
	If $E_G \cap E_H = \{12,13,14\}$ is a star, we have the following cases up to symmetry and inversion:
	\begin{enumerate}
		\item $E_H = E_G \cap E_H$ and $E_G = K_{1234}$. Here, $\M(G,H) = \M(G \cap H)^{-1}$.
		\item $E_H = E_G \cap E_H$ and $E_G = K_{1234} \setminus \{23\}$. Then
		\begin{align*}
			\M_1(G,H) &= \set{
				\begin{pmatrix}
					1 & 0 & \sigma_{13} & \sigma_{14} \\
					0 & 1 & 0 & 0 \\
					\sigma_{13} & 0 & 1 & 0 \\
					\sigma_{14} & 0 & 0 & 1
				\end{pmatrix}\colon
				\sigma_{13}^2 + \sigma_{14}^2 < 1} \\ &\cup \set{
				\begin{pmatrix}
					1 & \sigma_{12} & 0 & \sigma_{14} \\
					\sigma_{12} & 1 & 0 & 0 \\
					0 & 0 & 1 & 0 \\
					\sigma_{14} & 0 & 0 & 1
				\end{pmatrix}\colon
				\sigma_{12}^2 + \sigma_{14}^2 < 1},
		\end{align*}
		a union of two discs intersecting in a line segment.		
		\item $E_H = (E_G \cap E_H) \cup \{23\}$ and $E_G = K_{1234} \setminus \{23\}$. Then
		\begin{equation*}
			\M_1(G,H) = \set{
				\begin{pmatrix}
					1 & \sigma_{12} & \sigma_{13} & \sigma_{14} \\
					\sigma_{12} & 1 & \frac{\sigma_{12} \sigma_{13}}{1 - \sigma_{14}^2} & 0 \\
					\sigma_{13} & \frac{\sigma_{12} \sigma_{13}}{1 - \sigma_{14}^2}  & 1 & 0 \\
					\sigma_{14} & 0 & 0 & 1
				\end{pmatrix} \in \PD_4}.
		\end{equation*}
		\item $E_H = E_G \cap E_H$ and $E_G = K_{1234} \setminus \{23,24\}$. Then
		\begin{align*}
			\M_1(G,H) &= \set{
				\begin{pmatrix}
					1 & \sigma_{12} & 0 & 0 \\
					\sigma_{12} & 1 & 0 & 0 \\
					0 & 0 & 1 & 0 \\
					0 & 0 & 0 & 1
				\end{pmatrix}\colon
				\sigma_{12} \in (-1,1)} \\ &\cup \set{
				\begin{pmatrix}
					1 & 0 & \sigma_{13} & \sigma_{14} \\
					0 & 1 & 0 & 0 \\
					\sigma_{13} & 0 & 1 & 0 \\
					\sigma_{14} & 0 & 0 & 1
				\end{pmatrix}\colon
				\sigma_{13}^2 + \sigma_{14}^2 < 1}.
		\end{align*}
		\item $E_H = (E_G \cap E_H) \cup \{23\}$ and $E_G = K_{1234} \setminus \{23,24\}$. Then
		\begin{align*}
			\M_1(G,H) &= \set{
				\begin{pmatrix}
					1 & \sigma_{12} & \sigma_{13} & 0 \\
					\sigma_{12} & 1 & \sigma_{12} \sigma_{13} & 0 \\
					\sigma_{13} & \sigma_{12} \sigma_{13} & 1 & 0 \\
					0 & 0 & 0 & 1
				\end{pmatrix}\colon
				\sigma_{12}, \sigma_{13} \in (-1,1)} \\ &\cup \set{
				\begin{pmatrix}
					1 & 0 & \sigma_{13} & \sigma_{14} \\
					0 & 1 & 0 & 0 \\
					\sigma_{13} & 0 & 1 & 0 \\
					\sigma_{14} & 0 & 0 & 1
				\end{pmatrix}\colon
				\sigma_{13}^2 + \sigma_{14}^2 < 1}.
		\end{align*}
		\item $E_G = E_H = E_G \cap E_H$. Then $\M_1(G,H)$ is the union of the three coordinate line segments parametrized by $\sigma_{12}$, $\sigma_{13}$ and $\sigma_{14}$, all contained in $(-1,1)$.
	\end{enumerate}
\end{proposition}

In particular, for $|E_G \cap E_H| \leq 3$ the models $\M(G,H)$ are
connected in the Euclidean topology. Moreover, $\M(G,H)$ is smooth if
and only if its Zariski closure is irreducible if and only if
$\dim(\M_1(G,H)) = |E_G \cap E_H|$ is maximal. However,
\Cref{ex:pathConnSingular} below shows that there exist irreducible
but singular double Markovian models on four
vertices. \Cref{proposition:edgeIntersection3} also shows that double
Markovian models need not be equi-dimensional.

\begin{remark}
  The preceding classification for $|E_G \cap E_H| \leq 3$ and
  computer-assisted computations on $n \leq 4$ vertices show that in
  this range for every smooth model $\M(G,H)$ with $G \cap H$
  connected, there exist graphs $G',H'$ such that
  $\M(G,H) = \M(G',H')$ and $G' \cup H' = K_N$, thus the smoothness of
  the geometric model follows from \cref{theorem:Piotr}.  It is
  unknown whether there are smooth models without this property.
\end{remark}

\begin{remark}
  Any classification project would profit from a complete list of
  models of double Markovian CI structures or equivalently the set of
  completions of relations $\CIS{G,H}$.  There is no efficient,
  combinatorial algorithm to compute the completion of $\CIS{G,H}$,
  although computing the completion can be reduced to invocations of
  the Positivstellensatz and thus quantifier eliminiation
  \cite[Chapter~4]{bochnak2013real}.
  One such project would be to classify smoothness of double Markovian
  models on small vertex sets.  Exploiting \Cref{theorem:Aida}, this
  reduces to multiple smoothness queries for smaller models.  It seems
  like a worthwhile computational challenge to compile a table of the
  pairs of small connected graphs which have a smooth model.  We
  determined that for $n = 3,4,5$ there are $4 + 55 + 2644$ pairs of
  connected graphs which induce pairwise inequivalent CI~structures
  modulo isomorphy and duality.
  For more information on computations see
  \url{https://gaussoids.de/doublemarkov.html}.
\end{remark}

\section{Examples, counterexamples, and conjectures}
\label{sec:manyExamples}

\begin{example} \label{ex:CI_vs_corr_model} Consider the double
  Markovian CI structure arising from
  $G = H = \graph4{Eij,Eik,Ejl,Ekl}$, namely
  $\set{(14|), (14|23),(23|), (23|14)}$.  A computation in
  \texttt{Macaulay2} shows that (the Zariski closure of) the model
  $\M(G,H)$ has three irreducible components while it was determined
  in~\cite[Example~4.1]{drton2010smoothness} that the correlation
  model $\M_1(G,H) = \M(G,H) \cap \PD_{4,1}$ has four.  Thus there are
  algebraic differences between $\M_{1}(G,H)$ and~$\M(G,H)$.
\end{example}
\begin{example}\label{ex:pathConnSingular}
  Continuing \Cref{ex:SelfdualNonsmooth}, suppose that $G = H$ is the
  complete graph on $N = 12\cdots n$ minus the edge $12$. These double
  Markovian models are singular at the identity matrix but the models
  of all proper minors of $\CIS{G,G}$ are smooth.
  According to \cite[Proposition~4.2]{drton2010smoothness} with
  $C_1 = \emptyset$ and $C_2 = N \setminus 12$, the singular locus
  of this model is again a Gaussian CI~model and it is described as a
  submodel of $\M(G,G)$ by the CI~statements
  \begin{align*}
    &\text{$(12|)$ and $(12|C_2)$ from $\CIS{G,G}$}, \\
    &\text{$(1j|)$ and $(2j|)$ for all $j \in C_2$.} \tag{$*$} \label{eq:pathConnSingular:marginals}
  \end{align*}
  All but one of them are simple zero constraints on the covariance
  matrix.  By Schur complement, the remaining almost-principal minor
  equals
  \begin{align*}
    \det \Sigma_{12|C_2} = (\sigma_{12} -
      \Sigma_{1,C_2} \cdot \Sigma_{C_2}^{-1} \cdot \Sigma_{C_2,2}) \det \Sigma_{C_{2}}.
  \end{align*}
  By all the marginal independence statements in
  \eqref{eq:pathConnSingular:marginals}, the vectors $\Sigma_{1,C_2}$
  and $\Sigma_{C_2,2}$ as well as the entry $\sigma_{12}$ are zero, so
  the right-hand side of this equality vanishes, and $(12|C_2)$ is
  implied by the marginal statements.  Thus, the singular locus is in
  fact a \emph{linear subspace} of $\Sym^2(\RR^n)$ of codimension
  $2n-3$, intersected with~$\PD_n$.  This shows that double Markovian
  models can have singular loci of arbitrarily large dimension.

  It is instructive to compute the concrete case of $n=4$. The maximal
  possible dimension of the correlation model in this case is
  $|E_G| = 5$.  However, $\M_1(G,G)$ is of dimension $4 < 5$. Indeed,
  the only conditions on a positive definite matrix
  $\Sigma \in \M_1(G,G)$ are $\sigma_{12} = 0$ and
  $(\Sigma^{\rm adj})_{12} = 0$, which, using $\sigma_{12} = 0$,
  writes as
  \begin{equation*}
    f \defas \sigma_{13} (\sigma_{24} \sigma_{34} - \sigma_{23}) +  \sigma_{14} (\sigma_{23} \sigma_{34} - \sigma_{24}) = 0.
  \end{equation*}
  This is an irreducible polynomial, so the Zariski closure of
  $\M_1(G,G)$, which can be viewed as the vanishing set of the above
  polynomial inside the affine space where we forget the variable
  $\sigma_{12}$, is irreducible. In this case, the ideal $\SCI_{G,G}$
  is prime even without saturation, and it coincides with
  $\mathcal{I}(\M_1(G,G)) = \ideal{\sigma_{12}, f}$. Moreover, we can
  see again that $\M_1(G,G)$ is connected: for every positive definite
  matrix $\Sigma$ satisfying $\sigma_{12} = 0$ and $f = 0$, scale all
  variables \emph{except} for $\sigma_{34}$ by some $\varepsilon$
  tending to $0$.  This preserves the two equations and establishes a
  path inside $\M_1(G,G)$ connecting $\Sigma$ to a matrix with only
  the entry $\sigma_{34} \in (-1,1)$ possibly non-zero. The set of
  these matrices is clearly connected. As computed above, in the
  singular locus all variables but $\sigma_{34}$ are forced to zero,
  showing that it is a line inside $\PD_4$. In particular, failure of
  smoothness is not always due to reducibility for double Markovian
  models.
\end{example}

In \Cref{ex:pathConnSingular}, we have $G = H =$ an almost complete graph,
where only one edge is missing. This model is singular at the identity
matrix and therefore shows that the sufficient condition for smoothness
in \Cref{theorem:Piotr}, namely that $G \cup H = K_N$, cannot be weakened.
The singular locus in this example is a submodel described by the occurrence
of additional CI~statements. However, this is not always the case, as
\cite[Example~4.3]{drton2010smoothness} discovered.

\begin{example}\label{ex:MGcapHTooSmall}
  Let $G = \graph4{Eik, Eil, Ejk,Ejl,Ekl}$ and
  $H = \graph4{Eij,Eik, Eil, Ejk,Ejl}$.  The model $\M_1(G,H)$ agrees
  with its inverse model up to permutation of~$N$.  Moreover,
  $G \cup H = K_4$ is the complete graph, so this model is smooth of
  the expected dimension $|E_G \cap E_H| = 4$.  However, neither
  $\M_1(G,H)$ nor its inverse lie in the graphical model
  $\M(G \cap H)$.  Indeed, for every $\Sigma \in \M_1(G,H)$ the
  condition $(\Sigma^{-1})_{12} = 0$ translates into
  $\sigma_{12} = \sigma_{13} \sigma_{23} + \sigma_{14} \sigma_{24}$,
  in particular
  $\sigma_{12}, \sigma_{13}, \sigma_{14}, \sigma_{23}, \sigma_{24}$
  can all be non-zero at the same time.  The same is true for the
  inverse model, as it arises from the permutation exchanging $1$ with
  $4$ and $2$ with~$3$.
\end{example}

\begin{example}\label{ex:ZeroDimlModel}
  Consider the disjoint graphs $G = \graph4{Eik,Ejl}$ and
  $H = \graph4{Eij,Ejk,Ekl}$. The semi-definite model with ones on the
  diagonal consists of the five points
  \begin{equation*}
    \begin{pmatrix}
      1 & 0 & 0 & 0 \\
      0 & 1 & 0 & 0 \\
      0 & 0 & 1 & 0 \\
      0 & 0 & 0 & 1
    \end{pmatrix}, \ \
    \begin{pmatrix}
      1 & \pm 1 & 0 & 0 \\
      \pm 1 & 1 & 0 & 0 \\
      0 & 0 & 1 & \mp 1 \\
      0 & 0 & \mp 1 & 1
    \end{pmatrix},
  \end{equation*}
  where the signs of the first two rows and last two rows agree,
  respectively, but are independent of each other.  Thus semi-definite
  models can be disconnected even in dimension zero.
\end{example}

We have no analogue of \Cref{ex:ZeroDimlModel} with positive definite
matrices, and semi-definite models with no restriction on the diagonal
are connected as they are star-shaped with respect to the zero matrix.
The question if double Markovian models are connected in general is
open.
\begin{conjecture}\label{conj:M1GHconnected}
  For any $G,H$, the models $\M(G,H)$ and $\M_1(G,H)$ are connected.
\end{conjecture}

The two statements in \cref{conj:M1GHconnected} are equivalent.
Moreover, $\M_1(G,H) \setminus \set{ \mathbbm{1}_n }$ can be
disconnected as the $1$-dimensional smooth case shows.  Connectedness
of $\M_1(G,H)$ implies:
\begin{conjecture}\label{conj:M1GHsmoothIfMaxDim}
  The model $\M_1(G,H)$ is smooth if it has the maximal dimension
  $|E_G \cap E_H|$. Moreover, if $\M_1(G,H)$ is smooth, its Zariski
  closure is irreducible.
\end{conjecture}

\subsection*{Acknowledgements}
We thank Piotr Zwiernik for suggesting the problem and for helpful
discussions.  We also thank the organizers of the summer research
project ``Linear Spaces of Symmetric Matrices'' at MPI MiS Leipzig,
where this problem was first posed.  Aida Maraj suggested
\cref{theorem:Aida} to us as a generalization of a special case which
appeared in the first arXiv posting of this paper.  The first three
authors were supported by the Deutsche Forschungsgemeinschaft (DFG,
German Research Foundation) -- 314838170, GRK~2297 MathCoRe.  Frank
R\"ottger was supported by the Swiss National Science Foundation
(SNSF) -- grant 186858.

\bibliographystyle{amsplain}
\bibliography{question48}

\bigskip \medskip

\noindent
\footnotesize {\bf Authors' addresses:}

\smallskip

\noindent Tobias Boege, MPI-MiS Leipzig, Germany,
{\tt post@taboege.de}

\noindent Thomas Kahle, OvGU Magdeburg, Germany,
{\tt thomas.kahle@ovgu.de}

\noindent Andreas Kretschmer, OvGU Magdeburg, Germany,
{\tt andreas.kretschmer@ovgu.de}

\noindent Frank Röttger, Université de Genève, Switzerland,
{\tt frank.roettger@unige.ch}

\end{document}